\documentclass[12pt]{amsart}

\usepackage{amsmath,amssymb, amsthm,amscd,amsfonts,graphicx,fancyhdr}
\usepackage[marginratio=1:1,margin=.8in]{geometry}
\usepackage{url, hyperref, color, verbatim}%,natbib}%, ulem}
\usepackage{mathtools}
\usepackage{tikz}
\usetikzlibrary{matrix}
\usepackage{tabularx}
\usepackage{booktabs}
\usepackage{bbm}

\usepackage{cleveref}  % so use.e.g.,  \Cref{codim}

\DeclareMathOperator{\Ann}{Ann}

\newcommand{\fm}{\mathfrak m}

\theoremstyle{plain} %\numberwithin{equation}{section}
\newtheorem{thm}{Theorem}[section]
\newtheorem{corollary}[thm]{Corollary}

\newtheorem{cor}[thm]{Corollary}
\newtheorem{conj}[thm]{Conjecture}
\newtheorem{lm}[thm]{Lemma}
\newtheorem{proposition}[thm]{Proposition}

\theoremstyle{definition}
\newtheorem{defn}[thm]{Definition}

\newtheorem{remark}[thm]{Remark}
\newtheorem{example}[thm]{Example}

\newtheorem{question}[thm]{Question}

\usepackage{tikz}
\usepackage{stmaryrd}

\newcommand{\expect}[1]{\mathbb{E}\left[#1\right]}
\newcommand{\rmi}{{I}}
\newcommand{\ideal}[1]{\left<{#1}\right>}
\newcommand{\rmidist}{\mathcal{I}}

\newcommand{\prob}[1]{\mathbb{P}(#1)}
\newcommand{\h}{\underline{h}}
\newcommand{\R}{\mathbb{R}}
\newcommand{\soc}[1]{\mathrm{Soc}(#1)}
\newcommand{\mon}{\mathrm{Mon}}

\newcommand{\N}{\mathbb{N}}
\newcommand{\er}{Erd\"os-R\'enyi}

\begin{document}

\title[WLP and randomness]{The Weak Lefschetz property and unimodality of Hilbert functions of random monomial algebras}

\author[U.\ Nagel]{Uwe Nagel} 
\address{Department of
  Mathematics, University of Kentucky, 715 Patterson Office Tower,
  Lexington, KY 40506-0027, USA} 
\email{uwe.nagel@uky.edu} 
    
\author{Sonja Petrovi\'c}
\address{Department of Applied Mathematics, Illinois Institute of Technology, Chicago, IL 60616, USA} 
\email{sonja.petrovic@iit.edu}

%\author[Nagel and Petrovi\'c]{Uwe Nagel and  Sonja Petrovi\'c}\thanks{uwe.nagel@uky.edu, University of Kentucky. sonja.petrovic@iit.edu, Illinois Institute of Technology. SP is partially supported by the Simons Foundation's Travel Support for Mathematicians Gift 854770 and  DOE/SC award number \#1010629.} 

\begin{abstract}

In this work, we investigate  the presence of the weak Lefschetz property (WLP) and Hilbert functions for various types of random standard graded Artinian algebras. If an algebra has the WLP then its Hilbert function is unimodal. 

Using probabilistic models for random monomial algebras, our results and simulations suggest that in each considered regime  the Hilbert functions of the produced  algebras are unimodal with high probability. The WLP appears to be present with high probability most of the time. However, we propose that there is one scenario where the generated algebras fail to have the WLP with high probability.

%Somewhere??? We hope to initiate the study of unimodality of Hilbert functions and the weak Lefschetz property using probabilistic methods. 
\end{abstract}

\maketitle

%%%%%%%%%%%%%%%%%%%%%%%%%%%%%

\section{Introduction}

An Artinian graded algebra $A$ over a field is said to have the weak Lefschetz property (WLP) if there is a linear form such that multiplication by it induces, for each integer $j$,  maps $[A]_{j-1} \to [A]_j$ between consecutive degree components that always have maximal rank, i.e., are injective or surjective. The name is derived from the Hard Lefschetz Theorem, which establishes the WLP for the cohomology ring of a smooth manifold.  The Hilbert function of an algebra having the weak Lefschetz property is subject to strict constraints; for example, it must be unimodal, meaning that it cannot strictly increase after strictly decreasing once. 

Although the presence of the WLP can be expressed as a linear algebra problem, deciding it has proven to be very challenging in many cases. There are numerous results---but also just as many conjectures---about whether WLP holds for a given standard graded Artinian algebra (see, e.g., \cite{AHK, APP, APPS, BMMN-22, BMMN-23, BL, FMM, KX}).  In this paper, we propose a new point of view. We investigate the conditions under which this property holds by using probabilistic models for level and standard  graded Artinian algebras. We rely on randomness and probabilistic models because the main goal is to understand ``average" (or typical) behavior of monomial algebras, with respect to the weak Lefschetz property.   

%
%{\color{magenta}Opening story... } 
%The weak Lefschetz property (WLP) means that the map defined by the multiplication by a general linear form has maximal rank \cite{HMNW}. The $h$-vector of an algebra having the weak Lefschetz property is subject to strict constraints; for example, it must be  unimodal, meaning that it cannot strictly increase after  strictly decreasing once. 
%There are numerous results---but also just as many conjectures---about whether WLP holds for a given standard graded Artinian algebra (see, e.g., \cite{AHK, APP, APPS, BMMN-22, BMMN-23, BL, FMM, KX}).  In this paper, we investigate the conditions under which this property holds by using probabilistic models for level and standard  graded Artinian algebras. We rely on randomness and probabilistic models because the main goal is to understand ``average" (or: typical) behavior of monomial algebras, with respect to the Weak Lefschetz property.   

 The use of probabilistic methods has been pioneered in graph theory (see, e.g., \cite{ER, Gilbert}). The Erd\"os-R\'enyi model has been adapted in \cite{dLPSSW} to introduce probabilistic models for random monomial ideals.  Main results therein provide precise probabilistic statements about various algebraic invariants of (coordinate rings of) monomial ideals:  the probability distributions, expectations and thresholds for events involving monomial ideals with given Hilbert function, Krull dimension, first graded Betti numbers. Here, we further adapt this approach to study the WLP and the unimodality of Hilbert functions of quotients of polynomial rings by monomial ideals. These algebras are either constructed by using three different regimes for producing the monomial ideals or as random level algebras by selecting monomials of the socle. The Hilbert functions of the latter algebras are also known  as pure $O$-sequences, see \cite{BMMNZ}. 
 
Multiplication by a general linear form on an algebra $A$ is always injective if it has positive depth. If $A$ is Artinian, injectivity cannot always be true, but one often expects that multiplication between consecutive degrees has always maximal rank, i.e., $A$ has the WLP.  Formal probabilistic models for randomly generated algebras offer a controlled and principled way to test this expectation by constructing samples of algebras  using various regimes. Our results and simulations in this work suggest that in each of the considered regimes the Hilbert functions of the produced algebras are unimodal \emph{with high probability}. The situation is more nuanced regarding the WLP. Although the WLP appears to be present with high probability most of the time, we propose in  \Cref{conj:failure WLP} that there is one scenario where the produced  algebras fail to have the WLP with high probability. Our point of view offers many open questions, and we hope this note will motivate further investigations. 

% A sample result gives {a specific regime for generating random level algebras}. We expect to be able to prove  that the $h$-vector is unimodal \emph{with high probability}, and that the \emph{expected $h$-vector} is unimodal. We simulate....... to indicate... and prove....  Similar results hold for pure $O$-sequences, which arise as  Hilbert functions of monomial level algebras \cite{BMMNZ}; the question of their unimodality is open in general.  We can also quantify when equigenerated monomial algebras  \cite{AB} have the weak Lefschetz property  with high probability, using our  canonical probabilistic model for random algebras. Results of this type have proved to be foundational in discrete areas (e.g. graph theory), and we see them under the same light in commutative algebra. For example, the power of the probabilistic approach is often expressed by showing the existence of an algebra with some special property  and quantifying the probability with which such an object exists. {\color{magenta}refer to  theorem?} 

We briefly describe the structure of this work. In the following section, we review the main definitions necessary for understanding the technical results and simulations. In \Cref{sec:random alg}, we consider three regimes for producing random Artinian algebras with monomial relations. For each regime, we determine the expected value of the Hilbert functions is any degree. Evaluations and simulations suggest that these expected Hilbert functions are always unimodal. However, \Cref{conj:failure WLP} identifies a scenario, where we propose the WLP fails almost surely. We continue by studying random monomial level algebras. In \Cref{sec-Pure O}, we consider their Hilbert functions, a.k.a.\ pure $O$-sequences. In \Cref{thm:unimodal level} we show that under a natural model for random level algebras, the expected Hilbert function is always log-concave, and so in particular unimodal.  In \Cref{sec:WLP level}, we argue that random monomial level algebras have the WLP with high probability if the probability $p$ for selecting a monomial in the socle in close to zero or one. However, our simulations do not seem to reveal a clear pattern  of thresholds for the emergence of the WLP for other values of $p$.

%%%%%%%%%%%%%%%%%%

\section{Background and notation: level algebras, weak Lefschetz property, randomness, and sampling}

Let $S=\mathbbm{k}[x_1,..,.x_n]$ be the graded polynomial ring in $n$ variables over an infinite field $\mathbbm{k}$ where $\deg(x_i)=1$ for each $i$.  We will be interested in properties of standard graded Artinian $S$-algebras
\[
A = S/I = \bigoplus_{d \ge 0} A_d.
\]  The \emph{Hilbert function} $\h_A:\mathbb{N}\to\mathbb{N}$ of $A$ is
\[
\h_A=(h_0=1,h_1,h_2,\ldots)
\] where $h_d=h_A(d)=\dim_{\mathbbm{k}}A_d$ is the $\mathbbm{k}$-vector space dimension of the $d$-th graded component of $A$.  The \emph{socle} of $A$ is the annihilator of the homogeneous maximal ideal $\mathfrak{m}\subseteq A$, i.e.
\[
\soc{A}=\mathrm{ann}_A(\mathfrak{m}).
\]  The socle of $A$ is thus a homogeneous ideal of $S$.  Letting $s_i$ denote the number of degree $i$ generators of $\soc{A}$, the \emph{socle vector} $s(A)$ of $A$ is given by
\[
s(A) = (s_0=0,s_1,s_2...).
\]  
Note that $s_0 = 0$ if $A \neq \mathbbm{k}$. 
Since $A$ is Artinian, there exists a smallest non-negative integer $e$ such that $h_e\not=0$ and $h_{d}=0$ for all $d>e$, so that in this case $\h_A$ is a finite sequence to which we refer as the $h$-vector of $A$.   Note also that $s_e=h_e$ and $s_{d}=0$ for all $d>e$.
The integer $e$ is called the \emph{socle degree} of $A$.  If the generators of $\soc A$ are concentrated in a single degree, i.e. if $s(A)$ has the form
\[
s(A) = (0,0,\ldots,0,s_e)
\] where $s_e>0$, then the algebra $A$ is said to be \emph{level}.  The \emph{type} of a level algebra is the integer $s_e$.  For example, $A$ is Gorenstein if and only if it is level of type 1.

\begin{defn} A standard graded Artinian $S$-algebra $A$ has the \emph{weak Lefschetz property} (WLP) if there exists a   linear form $L\in A_1$ with the property that the multiplication map
\[
\times L:A_d\to A_{d+1}
\] is full rank (i.e. is either injective or surjective) for every $d\ge0$.  The linear form $L$ is called a \emph{Lefschetz element} for $A$.
\end{defn}

The $h$-vector of an algebra having the weak Lefschetz property is subject to strict constraints (see \cite[Proposition 3.5]{HMNW}).  For example, if $A$ has WLP, then $\h_A$ must be \emph{unimodal}, i.e. it cannot strictly increase after a strict decrease.  For level algebras, this is also a consequence of the following result.

\begin{thm}[{\cite[Propostion 2.1]{MMN}}] Let $A$ be a standard graded Artinian algebra. Let $L\in A_1$ be a general linear form and for $d\ge0$ consider the map $\phi_d:A_{d}\to A_{d+1}$ defined by multiplication by $L$.
\begin{enumerate}
\item[a)] If $\phi_{d_0}$ is surjective for some $d_0$, then $\phi_d$ is surjective for all $d\ge d_0$.
\item[b)] If $A$ is level and $\phi_{d_0}$ is injective for some $d_0$, then $\phi_{d_0}$ is injective for all $d\le d_0$.
\item[c)] In particular, if $A$ is level and $h_{d_0}=h_{d_0+1}$ for some $d_0$, then $A$ has WLP if and only if $\phi_{d_0}$ is injective (and hence an isomorphism).
\end{enumerate}
\end{thm}

The next statement shows that in our setting, we can focus our attention on a single linear form $L$.

\begin{proposition}[{\cite[Propostion 2.2]{MMN}}] Suppose that $A=S/I$ where $I\subset S$ is an Artinian monomial ideal.  Then, $A$ has WLP if and only if $x_1+\cdots+x_n$ is a Lefschetz element for $A$.
\end{proposition}

\medskip
Let us also recall the basic notation for probabilistic models of monomial ideals. 
Inspired by the study of random graphs and simplicial complexes, \cite{dLPSSW} defines a formal class of probabilistic models of random monomial ideals as follows. 
Fix an integer $D$ and a parameter $p=p(n,D)$, $0\leq p\leq 1$. Construct a random set of monomial ideal generators $B$ by including, independently, with probability $p$ each non-constant monomial of total degree at most $D$ in $n$ variables. The resulting random monomial ideal is simply $I=\ideal{B}$, and if $B=\emptyset$, then we let $I=\ideal{0}$. 
The notation for this data-generating process uses $\mathcal B(n,D,p)$ to denote the resulting distribution on the sets of monomials. This distribution on monomial sets induces a distribution on the set of ideals. This is called the \er-type distribution on monomial ideals and is denoted by $\rmidist(n,D,p)$. 

When an ideal $I$ is generated using  a random process just described, we say the ideal $I$ is an observation of the random variable drawn from the \er-type distribution. We denote this using standard notation for distributions and random variables: $$\rmi\sim\rmidist(n,D,p).$$ We may also refer to  $\rmi$ as an \er\ random monomial ideal. 

It is important to note that the set $\rmidist(n,D,p)$ is a parametric family of probability distributions, one for each set of values of the parameters $n$, $D$, and $p$. 
For the other probabilistic notions, we follow standard random variable notation; for example,  given a (discrete) random variable $Y$, we denote its  probability under the \er\ model as $\prob{Y}$ and its expected value by $\expect{Y}$. Expectation, or mean, is a weighted average computed as follows: $\expect{Y}=\sum_y y\prob{Y=y}$.  An important property  we use repeatedly is that expectation is linear, that is, $\expect{X+cY} = \expect{X}+c\expect{Y}$ for any two random variables $X$ and $Y$ and any constant $c$.  
To compute the expected value of a function of a random variable, we use the following formula: $\expect{g(Y)}=\sum_{y:\prob{Y=y}>0}g(y)\prob{Y=y}$. 
%{\color{magenta} To add: I don't know how much more we have to say here yet. Definition of expectation? fact that it is linear? can't decide.} {\color{green} Perhaps linearity as we use it.} Expectation is linear. Define $E[g(X)]$. }

%%%%%%%%%%%%%%%

%\section{Warm-up simulation results -maybe disperse } \label{section:warmup}
To better understand how we generate samples of random ideal data and how we compute and report various statistics from those samples, let us look at a typical simulation and non-unimodal Hilbert function statistics displayed in Table~\ref{table:ERunimodalStats}. 
\begin{table}[!h]
	\begin{tabular}{  l  | l}
	parameter setting &  	 \\%estimated probability  \\
	$n=5$, $D=100$  &  $\hat\nu$\\ % $h$   non-unimodal \\ 
	\hline
	$p$=0.01 & 0 \\  
	$p$=0.02 & 0.0417  \\
	$p$=0.1&  0 \\ 
	$p$=0.2& 0.01\\
	$0.3\leq p\leq 0.9$ & 0\\
	\end{tabular}
	\vspace{3mm}

\caption{Unimodality of $h$-vectors of  zero-dimensional monomial ideals  drawn  from the \er\ model $\rmidist(5,100,p)$. The column $\hat \nu$ represents the estimated probability $\nu=\prob{h_A\mbox{ is non-unimodal}}$, where $A=S/I$ for $\rmi\sim\rmidist(5,100,p)$.}
 \label{table:ERunimodalStats} 
\end{table}

This table represents repeated simulations of  $N$ randomly generated \er\ monomial ideals.  In this paper, we usually set $N$ to be $100$, and repeated each simulation ten times. 
	For each ideal $I$ in the sample of $N$ ideals, we first compute the Krull dimension of $A=S/I$, and only retain the zero-dimensional ones. Given \cite[Corollary 1.2]{dLPSSW}, one needs to choose $p$ at least $1/D$ to ensure that the Krull dimension is zero with high probability.  
For $p=0.01$ and $p=0.02$ there were $17$ and $48$ zero-dimensional ideals, respectively.  
The table reports the \emph{estimated probability} that $h$ is non-unimodal; this is simply the proportion of the sampled ideals which have a non-unimodal Hilbert function. The proper notation for reporting this statistic---acknowledging that $I$ is random and thereby so is $S/I$---is that the sampled proportion is $\hat \nu$, representing the estimated probability of the event that $S/\rmi$ has a non-unimodal Hilbert function:  $\nu=\prob{h_{S/\rmi}\mbox{ is non-unimodal}}$.  %Since $I$ is random, the coordinate ring $A$ is also random, and so we are computing a probability of a random event.
Thus $\hat\nu$ is a point estimator of $\mathbb P_{n,D,p}(\h_A\mbox{ is unimodal})$, computed under the probability distribution of $\rmi \sim \rmidist(n,D,p)$. 
 When $p=0.02$, the value $\hat\nu=0.0417$ means that $4.17\%$ of the $48$ zero-dimensional ideals in the sample produce a non-unimodal Hilbert function. The value $4.17\%$ is the \emph{average} result of ten repeated simulations of $100$ ideals each. 
In each of the samples we generate, each non-unimodal Hilbert function corresponds to a unique ideal in the sample. 
Based on the simulation reported, we conjecture that the Hilbert function is almost surely unimodal. 

While the table only displays the results for the case of $n=5$ variables and maximum degree $D=100$,  other values produced very similar results: we have tested the cases $n=3,4,5,7$, varying $D=50,100,150,200$. The code used for generating samples relies on the {\tt RandomMonomialIdeals} package in {\tt Macaulay2}, and is made  available on the following page: \[\mbox{\url{https://github.com/Sondzus/RMI-Hilbert-WLP}.}\]
 
In general, throughout this manuscript, the reader may notice that we  often set parameter values  $p\in(0,1)$ to be powers of $1/D$; the reason for this is precisely the Krull dimension threshold result cited in the previous paragraph. 

%This is now about how we report statistics and how we generate data, so that the readers who understand can skip, and those to whom this is new can get more details. :) 
%We have run simulations using functions {\tt randomMonomialIdeals} in Macaulay2 code available at.... TO BE ADDED.  for $n=3,4,5,7$, varying $D=50,100,150,200$. 
%\paragraph{NOTE TO ADD:} the statistics reported actually reflect the  samples of zero-dim ideals, regardless of the pre-set N=100 sample size. This is all behind the ER samples generated..
 % and also varying $D$ so that the expected dimension of the ER monomial ideal is zero ($1/D,2/D,\dots,1$). 
%The expected Hilbert function is \emph{always} unimodal. {\color{magenta}do we want to single out this computational result in another way?} 
%Results of a typical computation of \er\ random monomial ideals that are zero-dimen\-sional are presented in Table~\ref{table:ERunimodalStats}, which only shows the data for $D=100$ and $n=5$. 
%The sample size is set to $N=100$ for each entry in the table. 
%For  low values of $p$, not all of the 100 ideals are zero dimensional, thus they are not counted in the effective sample size. Specifically, 

Perhaps the conjecture that Hilbert functions are almost always unimodal is not surprising. We go one step further and compute the \emph{expected} Hilbert function. Since $h_A(d)$ is a function $\rmi\sim\rmidist(n,D,p)$, for every degree $d$ the quantity $h_A(d)$ is random. So we think of the Hilbert function as an (a priori infinite) random vector. For each degree $d$, we can compute the expectation of $h_A(d)$, which we denote by  $\expect{h(d)}$. The expectation of the entire vector will be denoted by the shorthand $\expect{h}$. \label{expected Hilbert}

% it is a function of a random variable of the random ideal $I\sim\rmidist(n,D,p)$, the function $h_A(d)$ is a function h is a function of I. clarify: E[h] vs average of sample -- so "simulated expectation" 
In principle, the expected Hilbert function can be computed by taking the definition of expectation, where the weighted average is taken over nontrivial monomial ideals. We carry this out in \Cref{sec-Pure O} for one scenario.  In practice, we \emph{estimate} this quantity by sampling ideals from the \er\ model repeatedly and computing the \emph{sample average} of the Hilbert functions from ideals in the sample.  
In {\bf all} of the samples we tested, the expected Hilbert function was unimodal. 
Figure~\ref{fig:ERexpectedH} contains sample plots for this phenomenon, where each curve in the figure is an expected Hilbert function for a sample of size $N=100$ for fixed values of parameters $n,D,p$.
\begin{figure}[!h]
	\vspace{-3mm}
	\includegraphics[scale=.4]{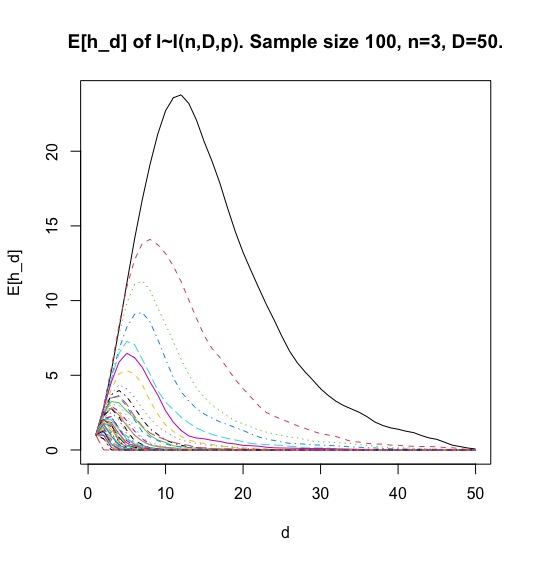}
	\includegraphics[scale=.4]{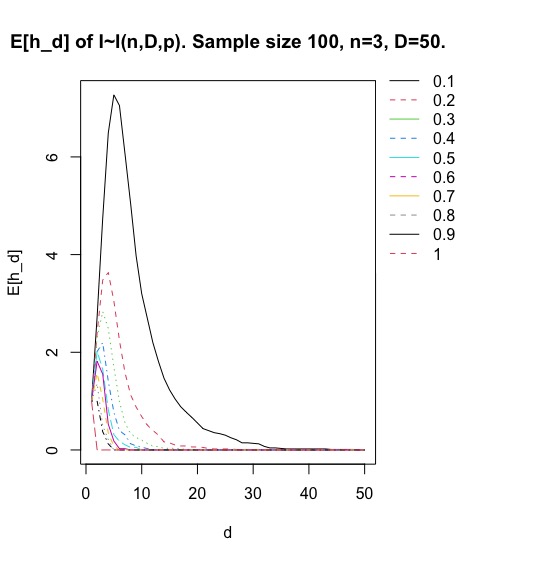}
	
	\vspace{-3mm}
	\includegraphics[scale=.4]{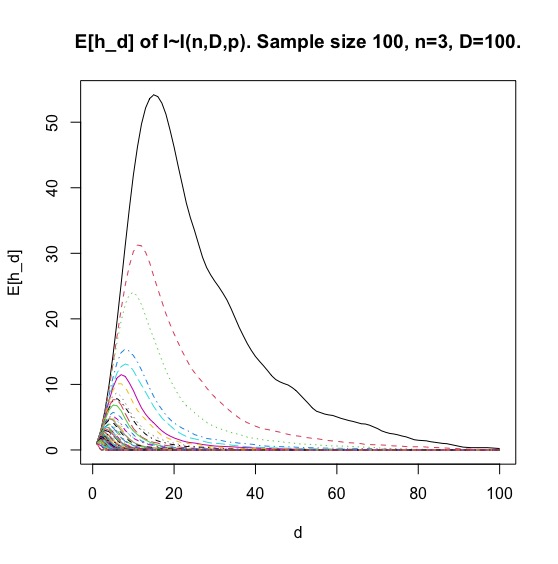}
	\includegraphics[scale=.4]{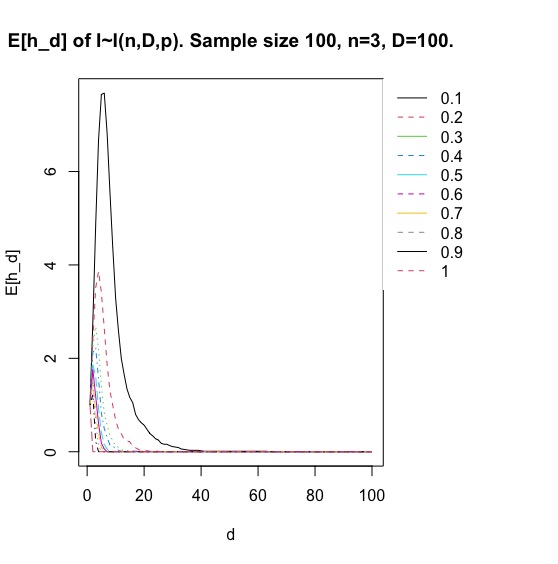}	
	\vspace{-3mm}
\caption{Expected Hilbert functions for $\rmi\sim\rmidist(n,D,p)$.  
%Top n=3 D=50; bottom n=3 D=100. 
Left: $D$ values of $p$ ranging from $1/D$ to $1$. Right: 10 selected values of $p$, with legend. Note that each curve,  of a fixed color, represents one sample of size $N=100$ for one value of $p$.  
%{\color{red}SONJA: I am giving up on this idea:  to contrast this with the formula (prop. \ref{prop:expected Hilb}; then also from thm 7.2 for the other case). Plotting the formula is too much work but I feel that the payoff is too low for that investment right now. Maybe in a revision!}
}
\label{fig:ERexpectedH}
\end{figure}
% {\color{green} Perhaps use $E[h]$ in the captions to avoid confusion with the expected value in a specific degree???}%{\color{magenta}sonja clarify. h is a function of I. clarify: E[h] vs average of sample -- so "simulated expectation" } 

\begin{figure}[!h]
	\vspace{-3mm}
	\includegraphics[scale=.4]{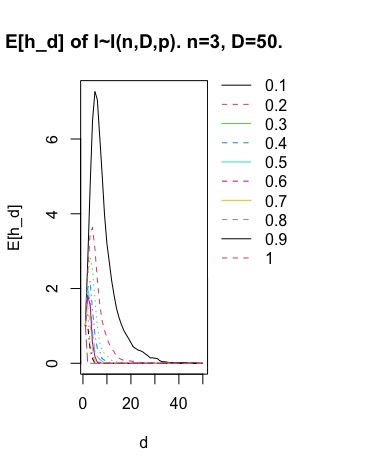}
	\includegraphics[scale=.4]{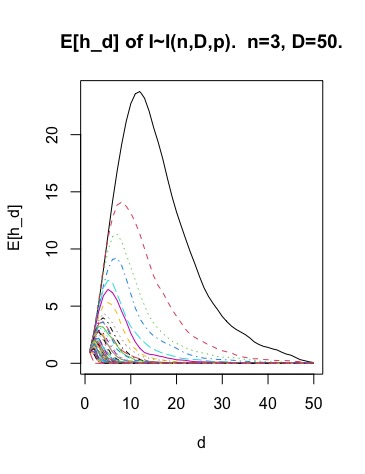}
	\includegraphics[scale=.4]{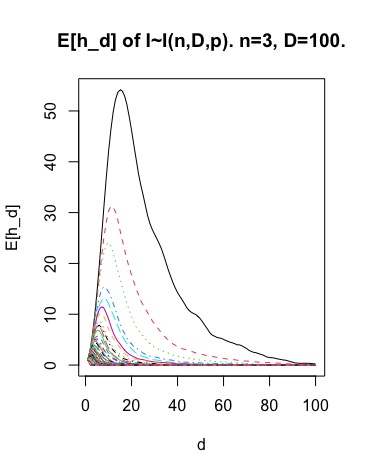}
	
	\vspace{-3mm}
	\includegraphics[scale=.4]{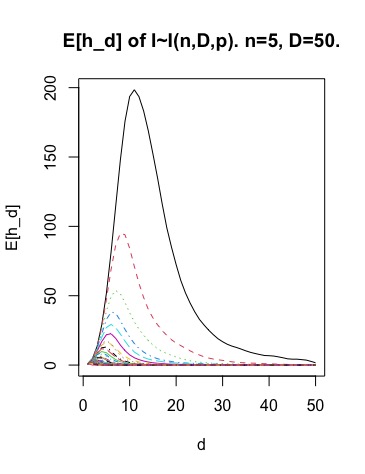}
	\includegraphics[scale=.4]{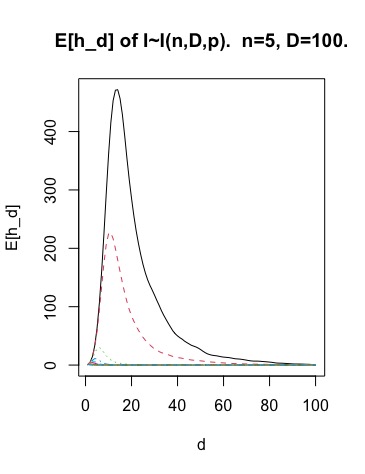}	
	\vspace{-3mm}

	\vspace{-3mm}
\caption{Expected Hilbert functions $h_d(S/\rmi)$ of  monomial ideals $\rmi\sim\rmidist(n,D,p)$.  
%Top n=3 D=50; bottom n=3 D=100. 
Each color curve represents one sample of size $N=100$ for one value of $p$. The top left figure shows the legend for values of the probability parameter $p$. The remaining four figures have about 20 curves so the legends are hidden.}
\label{fig:ERexpectedH}
\end{figure}

%%%%%%%%%%%%%%%

\section{Random standard graded Artinian algebras}
\label{sec:random alg}

We will consider three related regimes for producing random Artinian algebras. In each case we determine the expectation for the value of the Hilbert function in any degree and compare the result with simulations.

Let $S=\mathbbm{k}[x_1,\ldots,x_n]$ with its standard grading and let $J \subset S$ be a homogenous ideal such that $A = S/J$ is Artinian. If $n$ is at most two, then $A$ has the weak Lefschetz property and thus a unimodal Hilbert function (see \cite[Proposition 4.4]{HMNW}). However, if $n \ge 3$ then this is no longer true. 

\begin{example} (i)  The $h$-vector of
\[
A=S/\langle x_1x_2^{13},x_1x_2^{10}x_3,x_1^7x_2^8x_3^2,x_1^7x_3^3,x_1^4x_2x_3^4,x_2^7x_3^8,x_1x_3^{11},x_1^{20},x_2^{20},x_3^{20}\rangle
\] 
is
\begin{align*}
(3, 6, 10, 15, 21, 27, 33, 38, 40, 40, 39, 38, 37, 36, 35, 35,\\
35, 35, 36, 34, 31, 27, 23, 19, 15, 12, 9, 6, 4, 3, 2, 1),
\end{align*} which upon inspection is seen to be not unimodal.  One ``peak" occurs at $\h_A(9)=\h_A(10)=40$ and another occurs at $\h_A(19)=36$.

(ii)
The following ideal gives an example of an $h$-vector with 3 peaks:
%
%              7   6 36   4 6     3 5 2     13 2   6 5   3 6     10 7     7 9   29 18   7 26   2 5 27   2 29     6 35   41   50   50   50                                                                                                                              
%              
%o14 = ideal (x , x x  , x x x , x x x , x x  x , x x , x x , x x  x , x x x , x  x  , x x  , x x x  , x x  , x x x  , x  , x  , x  , x  )                                                                                                                             
%
%              1   1 2    1 2 3   1 2 3   1 2  3   1 3   1 3   1 2  3   1 2 3   2  3    2 3    1 2 3    1 3    1 2 3    3    1    2    3     
%
%
% why is this not coded as a single-line string?! 
\begin{align*}
\langle x_1^7 , x_1^6 x_2^{36}  , x_1^4 x_2^6x_3 , x_1^3 x_2^5 x_3^2 , x_1 x_2^{13} x_3^2 , x_1^6 x_3^5 , x_1^3 x_3^6 , x_1 x_2^{10}  x_3^7 , x_1 x_2^7 x_3^9 , 
\\
x_2^{29}  x_3^{18}  , x_2^7 x_3^{26} 
 x_1^2 x_2^5 x_3^{27}  , x_1^2 x_3^{29}  , x_1 x_2^6 x_3^{36}  , x_3^{41}  , x_1^{50}  , x_2^{50}  , x_3^{50}  \rangle
\end{align*}
has as $h$-vector the following: 

%\begin{verbatim}
%o15 = {
(1, 3, 6, 10, 15, 21, 28, 35, 42, 48, {\bf 52, 52}, 51, 49, 49, 50, {\bf 52, 52}, 50, 47, 45, 45, 46, 47, 48, 49, 50,
%      ------------------------------------------------------------------------------------------------------------
      51, 52, 53, {\bf 54, 54, 54}, 53, 51, 49, 49, 49, 49, 49, 49, 48, 44, 42, 40, 38, 36, 33, 32, 31, 29, 26, 22, 18,
%      ------------------------------------------------------------------------------------------------------------
      14, 12, 11, 10, 9, 8, 7, 6, 5, 4, 3, 2, 1). 
% \end{verbatim}
% \normalsize 
 \end{example}

Let $\rmi\sim \rmidist(n,D,p)$ be any \er-type random monomial ideal. We use it to produce an Artinian algebra in three ways. 

First, consider $A= S/I$, where we have to assume that $p > \frac{1}{D}$ in order to ensure that $A$ has almost surely dimension zero as $D \mapsto \infty$. 
We start by explicitly computing the probability that a monomial is not in the ideal $\rmi$ under the \er\ model.  We set $q = 1 -p$. 

% this result {\color{magenta}and this is not in the RMI paper? -- likely just hidden in a proof and we need it explicitly here--sonja} 
 
 \begin{proposition}
    \label{prop:not in ideal} 
If $\rmi\sim \rmidist(n,D,p)$ then, for every integer $d \ge 0$, one has 
 \begin{align} 
 \prob{x^{a} \not\in\rmi}  & = q^{N-1}, 
 \end{align}
 where 
 \begin{align*}
 N & = \sum_{j= 0}^{D} h_B(j)  = \begin{cases}
 (a_1 +1) \cdots (a_n +1) & \text{ if $|a| \le D$}; \\
 (a_1 +1) \cdots (a_n +1) - \sum_{j= D+1}^{|a|} h_B(j) & \text{ if $|a| \ge D$}, 
 \end{cases}
 \end{align*}
and $h_B$ denotes the Hilbert function of the complete intersection $B = S/\langle x_1^{a_1 + 1},\ldots,x_n^{a_n+1} \rangle $. 
 \end{proposition}
 
%{\bf Notation: $q=1-p$, check we defined it!}
 
 \begin{proof}
% For each monomial $x^{a}\in S$,  let $1_{a}$ be the indicator random variable for the event $x^{a}\not\in\rmi$ where $A=S/\rmi$.  
%Thus, $1_{a} = 1$ 
Note that a monomial $x^{a}$ is not in $\rmi$ if and only if every monomial $x^{b}\in S$ of degree at most $D$ that divides $x^{a}$ is not in $I$. 
Since only non-constant monomials are chosen as generators of $\rmi$,  it follows that 
\[
%\expect{ 1_{a}}  = \prob{1_{a}=1} = 
%\sum_{\substack{a \in \N_0^n \text{ s.t.}\\ |a| \le D }} 
\prob{x^{a} \not\in\rmi}  = q^{N-1}, 
\]
where $N$ is the cardinality of the set $M = \{ x^b \in S \; \mid \; |b| \le D \text{ and $x^b$ divides $x^a$}\}$. Observe that a monomial $x^b$ is in $M$ if and only if it is not in the ideal $J = \langle x_1^{a_1 + 1},\ldots,x_n^{a_n+1} \rangle $ and its degree is at most $D$. This proves $N = \sum_{j= 0}^{D} h_B(j)$. 
The second equality for $N$ follows from the well-known facts that 
\[
\deg J =  (a_1 +1) \cdots (a_n +1) = \sum_{j= 0}^{|a|} h_B(j). 
\]
and $h_B(j) = $ if $j > |a|$. 
 \end{proof}

\begin{remark}
The Hilbert function $h_B$ of a complete intersection $B = S/\langle x_1^{a_1 + 1},\ldots,x_n^{a_n+1} \rangle $ is encoded in the formula for its generating function
\[
\sum_{j= 0}^{|a|} h_B(j) z^j = \prod_{i=1}^n (1 + z + \cdots + z^{a_i}).  
\]
\end{remark}

The above result has the following consequence for the Hilbert function, where we use the convention that a sum 
$\sum_{j = a}^e f(j)$ is defined to be zero if $a > e$.  
 
 \begin{proposition}
   \label{prop:expected Hilb}
 If $\rmi \sim \rmidist(n,D,p)$ then the expectation of $\dim_K [S/\rmi]_d$ with $d \ge 1$ is 
 \[
 \expect{h_d}  = q^{-1} \cdot  \sum_{\substack{a \in \N_0^n \text{ s.t.}\\ |a|=d}} 
q^{ (a_1 +1) \cdots (a_n +1) - \sum_{j= D+1}^{|a|} h_{B_a} (j)}, 
%  q^d\sum_{\substack{(a_1,a_2,a_3)\text{ s.t.}\\a_1+a_2+a_3=d}}q^{a_1a_2a_3+a_1a_2+a_1a_3+a_2a_3}. 
\]
where  $B_a = S/\langle x_1^{a_1 + 1},\ldots,x_n^{a_n+1} \rangle $. 
 \end{proposition}

 \begin{proof} 
For each monomial $x^{a}\in S$,  let $1_{a}$ be the indicator random variable for the event $x^{a}\not\in\rmi$ where $A=S/\rmi$.  Since $h_d=\sum_{x^{\alpha}\in S_d} 1_{\alpha}$,  linearity of expectation and \Cref{prop:not in ideal} implies our assertion. 
\end{proof}

Second, consider the Artinian standard graded algebra  
\[
A =  S/(\rmi+\langle x_1^{D+1},\ldots,x_n^{D+1}\rangle).
 \]  
In his case, one has to adjust the above formula as follows. 

%\[
%A = S/(\rmi+\langle x_1^{D+1},\ldots,x_n^{D+1}\rangle)=\bigoplus_{i=1}^r A_i.
%\] 
% 
%  
%The above formula becomes simpler if one considers the $h$-vectors of the following monomial algebras: 
  
\begin{corollary} 
\label{cor:add powers}
Let $J = \langle x_1^{D+1},\ldots,x_n^{D+1}\rangle$. 
  If $\rmi \sim \rmidist(n,D,p)$ then the expectation of $\dim_K [S/(\rmi + J)]_d$  is 
  \[
 \expect{h_d}  = q^{-1} \cdot  \sum_{\substack{x^a \in [S]_d \setminus J} } % \text{ s.t.}\\ |a|=d}} 
q^{ (a_1 +1) \cdots (a_n +1) - \sum_{j= D+1}^{|a|} h_{B_a} (j)}, 
\]
where  $B_a = S/\langle x_1^{a_1 + 1},\ldots,x_n^{a_n+1} \rangle $.
%, $C = S/\langle x_1^{D+1},\ldots,x_n^{D+1})\rangle$ and 
%\[
%t = \max \Big \{0,\ \binom{n + |a|}{n} - \binom{n + D}{n} \Big \} - 1 -  \sum_{j = D+1}^{|a|} h_C (j). 
%\]
 \end{corollary} 
 
 \begin{proof}
 We argue similarly as above. 
Note that  $\dim_K [S/(\rmi + J)]_d$ is the number of degree $d$ monomials $x^a \in S$ that are not in $I + J$. If $x^a$ is not in $J$, then $x^a \notin I+J$ if and only if $x^a \notin I$. The probability of the latter event is given by \Cref{prop:not in ideal}. Hence the claim follows by linearity of expectaction.  
 \end{proof}

Observe that in particular one has $ \expect{h_d} = 0$ if $d > nD$ because then $[J]_d = [S]_d$. 
 
A third option  is to consider 
\[
 A=S/(\rmi+\langle x_1,\ldots,x_n\rangle ^{D+1}).
 \] 
In this case, one has $[A]_d = 0$ if $d > D$. Thus, it is enough to consider degrees $d \le D$, and \Cref{prop:expected Hilb} gives. 
 
 \begin{corollary} 
   \label{cor:add power max ideal}
  If $\rmi \sim \rmidist(n,D,p)$ then the expectation of $\dim_K [S/(\rmi + \langle x_1,\ldots,x_n\rangle ^{D+1}]_d$ with $1<d \le D$ is 
  \[
 \expect{h_d}  = q^{-1} \cdot  \sum_{\substack{a \in \N_0^n \text{ s.t.}\\ |a|=d}} 
q^{ (a_1 +1) \cdots (a_n +1)}. 
\]
 \end{corollary}
 
% \begin{problem} Find ways to understand/estimate the above sum .
% \end{problem}
 
  \begin{example} Let $n=3$ and consider the principal ideal $I=\langle x_1^{a_1}x_2^{a_2}x_3^{a_3}\rangle$ where $a_1,a_2,a_3>0$ and $a_1+a_2+a_3=d$. Set $A=S/(I+\langle x_1,x_2,x_3\rangle ^D)$ where $D\gg d$.  Then the $h$-vector of $A$ is has the form
 \[
(1,3,\ldots,\binom{d-1}{2},\binom{d+2}{2}-1,\binom{d+3}{2}-3,\ldots,\binom{d+j+2}{2}-\binom{j+2}{2},\ldots,\binom{D+1}{2}-\binom{D-d+1}{2}),
 \] 
 which is easily seen to be strictly increasing.
 \end{example}
 
 \begin{question} Does there exist an integer $D$ and choice of probability parameter $p$ so that the expected $h$-vector
 \[
 \expect{\h_A}=(1,\expect{h_1},\ldots,\expect{h_{D-1}})
 \] is \emph{not} unimodal in one of the three regimes above?
 \end{question}
 
% {\color{magenta}
%  for the model  
% \(
% S/(\rmi+\langle x_1^{D+1},\ldots,x_n^{D+1}\rangle).
% \)
% THE SIMULATIONS indicate the answer seems to be no. !    see section 5.1 and tables.
%} 
 Simulation results  for $n=3,4,5$ suggest that the answer to this question is a firm `no'.  It is an open problem to prove this. 
 
 %%%%%%%%%%%%%%%%%%%%%%%%%%%
 
%{\color{magenta} Moved from the very end of the paper and shortened - okay?}

There are other regimes for producing random algebras. For example, one can consider algebras $S/I$, where the generators of $I$ are sampled from the monomials in $S$ of fixed degree $D$.  Results in \cite{AB} are relevant in this case.  We leave the investigation of this and further regimes for future work.  

%%%%%%

\subsection{Simulation results}

%See the paragraph we wrote about Macaulay2 code available at...TBD. 
%Two more functions for the artinian setting: {\tt randomArtinAlgebras},  and {\tt randomArtinPlusAlgebras} 

Table \ref{table:RandomArtinUnimodalStats} shows statistics from a typical simulation for random Artinian algebras with $I+(x_1^D,\dots,x_n^D)$ for $\rmi \sim \rmidist(n,D,p)$. 
The tables report the estimated probability $\hat\nu$ that $h$ is non-unimodal. In other words, $\hat\nu$ is a point estimator of $\mathbb P_{n,D,p}(\h_A\mbox{ is unimodal})$, computed under the probability distribution of $I+(x_1^D,\dots,x_n^D)$ with $\rmi \sim \rmidist(n,D,p)$. 
In each table, we fix parameters $n$ and $D$ and the sample size $N$, and we vary the probability parameter $p$.   We then repeat the experiment of sampling $N$ random artin algebras ten times over, and report the average over those ten experiments, each with sample size $N$. There was very little variation between the repetitions, suggesting that the chosen $N$ suffices to capture large-sample properties of the random algebras. 

Given the threshold results in \cite[Corollary 1.2]{dLPSSW}, choosing the probability parameter $p$ to be $1/D^i$ for $i=1\dots n$ varies the expected Krull dimension of the random ideal $\rmi$. This is how the specific values of $p$ were chosen in the simulation. 
Perhaps the most interesting phenomenon we noticed is that around the value $p=1/D^2$ and $1/D^3$, there seems to be an increase in the number of non-unimodal $h$-vectors in the sample of random algebras.

\begin{table}[h]

In the two tables below, $\hat \nu$ represents the estimated probability $\nu=\prob{h_A\mbox{ is non-unimodal}}$. 

\medskip
 
	\begin{tabular}{  l | l}
		parameter setting  & \\% estimated probability  \\ 
		$n=3$, $D=100$ &  $\hat\nu$\\ %	  $h$   non-unimodal \\ 
	\hline
	 $p= 0.000001$&  0 \\ 
	 $p= 0.0001$ & 0.26  \\
	$p= 0.01$ & 0.6 \\ %6 out of 100 ! in repeated simulations (repeated them 5 times).   
	 $p= 0.1$ &  0.1 \\ 
	 $0.2\leq p\leq 0.9 $& 0 \\ 
	\end{tabular}
%
%	\vspace{5mm}
%
	\qquad 
	\begin{tabular}{  l | l}
	parameter setting &  	 \\%estimated probability  \\
	$n=4$, $D=50$  &  $\hat\nu$\\ % $h$   non-unimodal \\ 
	\hline
	$p=1/D^3=0.000008$ & 0 \\
	$p=1/D^2= $0.0004 & 0.26  \\
%	$p= $0.001 & 0.11 \\ %?TBD  \\ 	
	$p= 1/D=0.02$ & 0.06 \\  
	 $0.1\leq p\leq 0.9 $&  0 \\ 
	\end{tabular}
	\qquad 
	\begin{tabular}{  l | l}
%	parameter setting&  	 estimated probability  \\
%	 $n=5$, $D=50$  &  $h$   non-unimodal \\ 
	parameter setting&  	 \\
	 $n=5$, $D=50$  &  $\hat\nu$\\ 
	\hline
	$p=0.000008$ & 0 \\
	$p= 0.0004$ & 0.2 \\ %Number of non-unimodal h-vectors in the sample: 2 -- for D50 n5 p.0004 -- this is N=100.
	$p= 0.001$ & 0.4 \\ 	%Number of non-unimodal h-vectors in the sample: 2 -- for D50 n5 p.001
	$p=0.02$ & 0.01 \\
	$p= 0.1$ & 0.01\\
	$0.2\leq p\leq 0.9 $ & 0\\
	\end{tabular}

	\vspace{3mm}

\caption{The estimated probability $\hat\nu$ that $h_A$ is non-unimodal, for samples of random  artin algebras  $\rmi+(x_1^D,\dots,x_n^D)$, where $\rmi \sim \rmidist(n,D,p)$. 
Sample size for each table entry is $N=100$ for $n=3$ variables, and $N=50$ for $4$ and $5$ variables for computational reasons. %, because it takes much longer to check the WLP for more than $3$ variables. 
Each value in the table is the average proportion for the given sample size obtained from \emph{ten repeated simulations} for the fixed values of parameters $n,D,p,N$. 
}
 \label{table:RandomArtinUnimodalStats} 
\end{table}

Table \ref{table:RandomArtinPlusUnimodalStats} shows statistics from a typical simulation for random ER monomial ideal with powers of variables added, or power of maximal ideal added. These are random Artinian algebras with $\rmi+(x_1,\dots,x_n)^D$ for $\rmi \sim \rmidist(n,D,p)$.  
\begin{table}[!h]
In the tables below, $\hat \nu$ represents the estimated probability $\nu=\prob{h_A\mbox{ is non-unimodal}}$. 

\medskip
 
	\begin{tabular}{  l| l}
	$n=3$, $D=50$ &  \\
	parameter setting &  $\hat\nu$\\ %est.prob. non-unimodal \\ 
	\hline
	 $p=$0.000008&  0 \\ 
	 $p=$0.0004 & 0.1  \\ % 5 out of 50, or 1/50, or 6/50, 2/50.  take average. 
	 $p=$0.02 & 0.04 \\  
	 $p=$0.01 &  0.16 \\   %7/50 or 9/50 etc. 
	 $p=$0.1& 0.1\\ %sometimes 1 and sometimes 0 out of 50  but more often 0?  so report "<0.1"? naah.
	 $p= $0.2,\dots,0.9&0\\
	 \end{tabular}
\qquad
	\begin{tabular}{  l | l}
	$n=3$, $D=100$&  \\
	parameter setting &  $\hat\nu$\\ %est.prob. non-unimodal \\ 
	\hline
	 $p=$0.000001&  0 \\ 
	 $p=$0.0001 & 0.24  \\
	 $p=$0.001& 0.4 \\ %6 out of 50 on avg.  D=100
	 $p=$0.01 & 0.12 \\  
%Number of non-unimodal h-vectors in the sample: 0 -- for D50 n3 p.05
%Number of non-unimodal h-vectors in the sample: 2 -- for D50 n3 p.05 for N=100 so this would be uhat = 0.2 ! 
	 $p=$0.1&0.4 \\ % 2 out of 50 on averageD=100
	 $p= $0.2,\dots,0.9&0\\%D=100
	\end{tabular}
%
%	\vspace{5mm}
%
\qquad
	\begin{tabular}{  l  |l}
	$n=4$, $D=50$ &  \\
	parameter setting &  $\hat\nu$\\ %est.prob. non-unimodal \\ 
	\hline
	$p=$0.0004& 0.16 \\ %8/50
	$p=$0.001&  0.1* \\  %Number of non-unimodal h-vectors in the sample: 1 -- for D50 n4 p.001 -- changed N=10 for computational ease
	$p=$0.02& 0.4 \\ 
	$p=$0.1& 0.4\\ 
	$p=$0.2& 0\\ 
	$p=$0.3& 0.2\\ 
	$p=$0.14,$\dots$,0.9& 0\\ 
	\end{tabular}
		\vspace{3mm}

	\caption{The estimated probability $\hat\nu$ that $h_A$ is non-unimodal, for samples of random  artin algebras  $I+(x_1,\dots,x_n)^D$, where $\rmi \sim \rmidist(n,D,p)$. 
Each value in the table is the average proportion for the given sample size obtained from \emph{five to ten repeated simulations} for the fixed values of parameters $n,D,p,N$. The sample size $N$ was set to $100$, except in the entries marked by $^*$ where the sample size was $20$ or $50$ in repeated simulations.  
}
 \label{table:RandomArtinPlusUnimodalStats} 
\end{table}

%{\color{magenta}sonja's note:}  after the next revision of this draft, I may consider doing more simulations to get more data with larger sample sizes. The computation time is a real problem and Macaulay2 gets tired of storing all these ideals fairly quickly. ...  

\subsection{Expected Hilbert functions}

In \Cref{prop:expected Hilb} and Corollaries \ref{cor:add powers}, \ref{cor:add power max ideal}, we determined the expected $h$-vector, $\expect{\h_A}$, of the random algebras considered above. Evaluating these formulas, it turned out that the expected $h$-vector was \emph{always unimodal}, for all values of $n,D,p,N$ we have tried.  Various figures in this section illustrate this phenomenon. In particular, this includes all the cases corresponding to Tables \ref{table:RandomArtinUnimodalStats} and \ref{table:RandomArtinPlusUnimodalStats}. 
%{\color{magenta} Thanks for checking! Did you also try to evaluate the formula in the first case, i.e., using \Cref{prop:expected Hilb}? Again always unimodal???}  {\color{cyan}Well, I did not evaluate the formula in the first case, sorry.} 
%\\ {\color{magenta} 
%Other issue: I replaced ``artin algebra" by ``Artinian algebra" everywhere except at the headings of individual tables as in Figure 3. Where would one make that change? - The same issues arises elsewhere. ??? } {\color{cyan}In R. I did it. Remove this and remaining comments once you approve of the changes!:)} 

\begin{figure}[!h]
Random Artinian  algebras of the form $\rmi+(x_1^D,\dots,x_n^D)$ for  $n=3, D=50$: 

	\includegraphics[scale=.4]{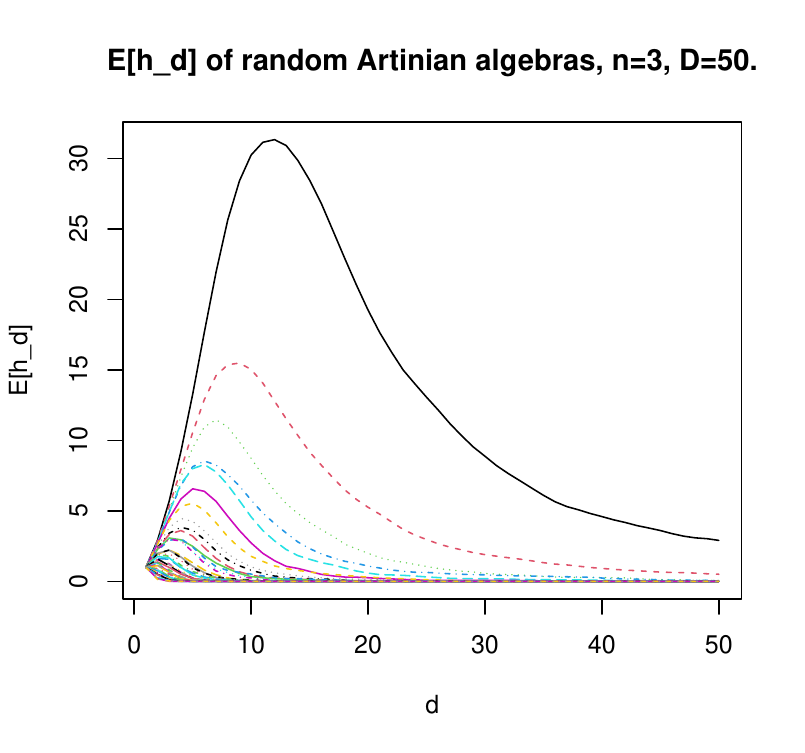}
	\includegraphics[scale=.4]{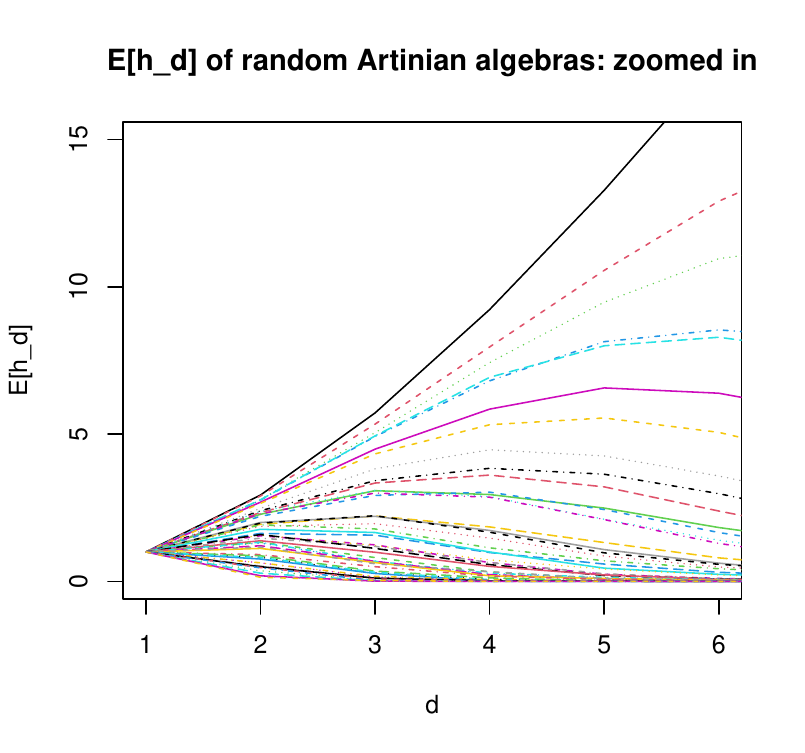}
	\includegraphics[scale=.4]{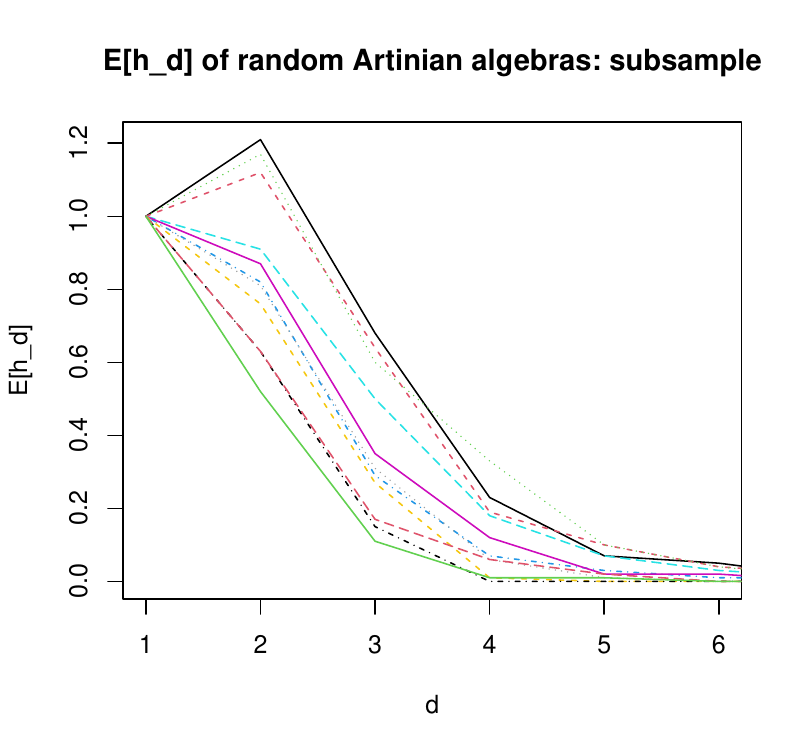}
\caption{Expected Hilbert functions Artinian algebras $\rmi+(x_1^D,\dots,x_n^D)$, where $\rmi \sim \rmidist(n,D,p)$.  Since the figure on the left gives a global trend, the middle and right figure offer a zoomed-in view of the sample (note the truncated $y$-axes in the middle and right figures). 
}
\label{fig:expectedH random artin algebras}
\end{figure}

\begin{figure}[!h]
Random Artinian  algebras of the form $\rmi+(x_1^D,\dots,x_n^D)$ for  $n=3, D=100$: 

	\includegraphics[scale=.4]{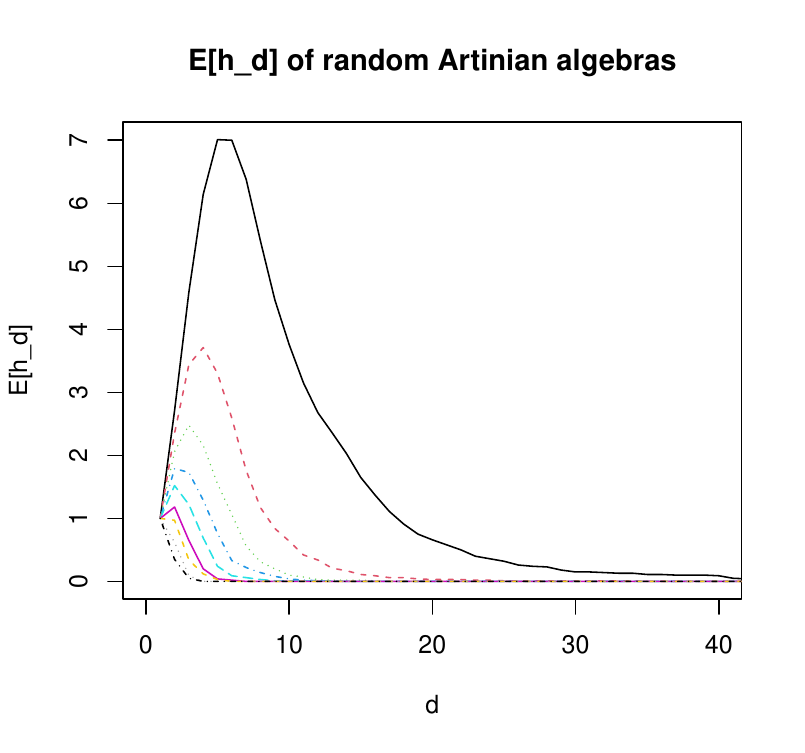}
	\includegraphics[scale=.4]{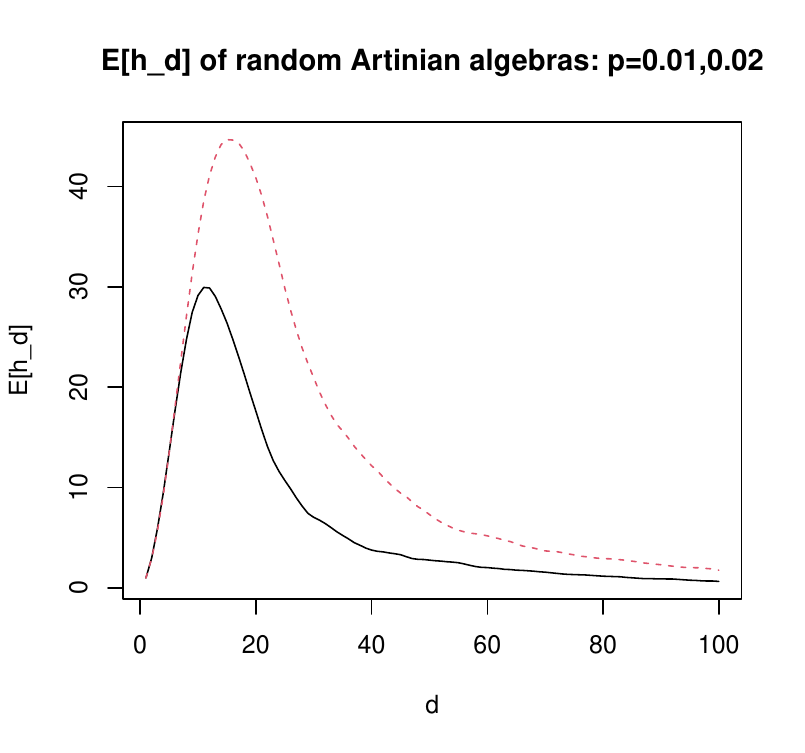}

\caption{Expected Hilbert functions for select random Artinian  algebras $\rmi+(x_1^D,\dots,x_n^D)$, where $\rmi \sim \rmidist(n,D,p)$, obtained from  the same samples of monomial algebras from Table~\ref{table:RandomArtinUnimodalStats}. The figure on the right represents 50 algebras generated using the values $p=0.01$ and $p=0.02$, while the figure on the left values $p>0.1$.    %{\color{magenta} Should it be $h_d$ instead of $E[h_d]$ in the above table?}{\color{cyan}No, it is the expected h.} {\color{magenta} Perhaps specify the choices of $p$ instead of the vaguer ``two types"??? }{\color{cyan}You are right. I looked back, could not find what the values were with certainty, so had to recreate some plots.}
}
\label{fig:expectedH random artin algebras - 2}
\end{figure}

\begin{figure}[!h]
Random Artinian algebras of the form $\rmi+(x_1^D,\dots,x_n^D)$ for  $n=4, D=50$: 

	\includegraphics[scale=.4]{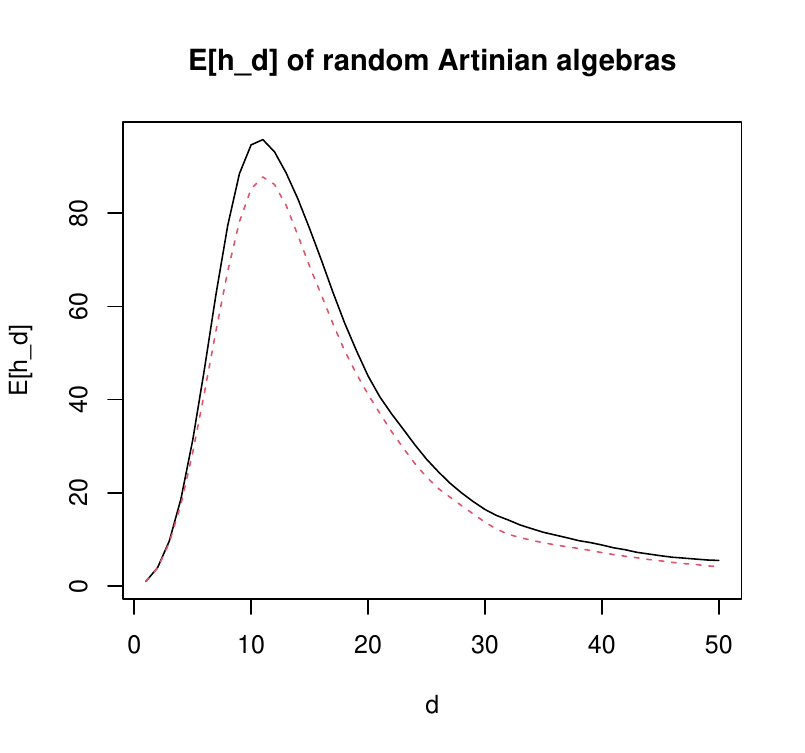}
	\includegraphics[scale=.4]{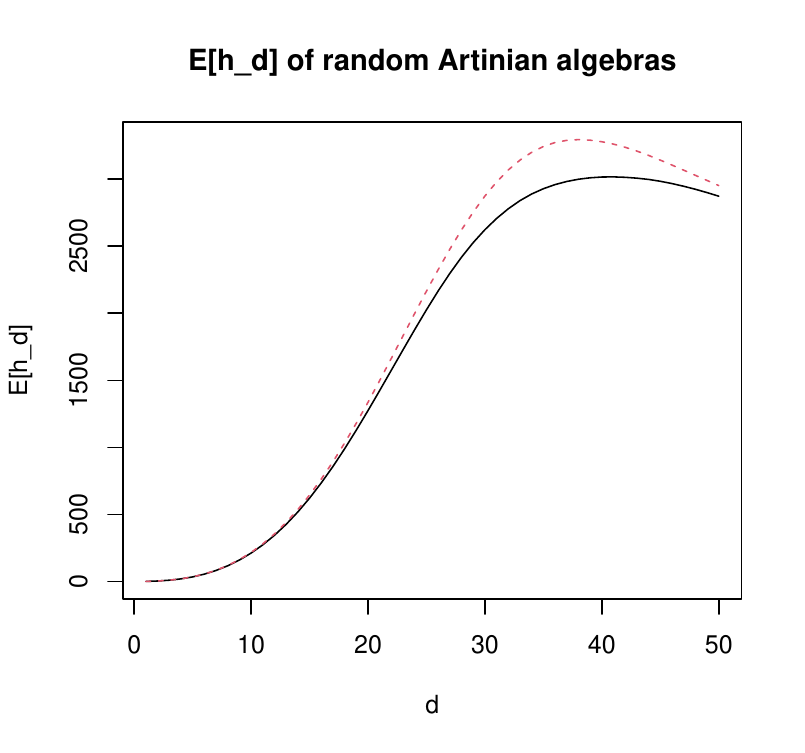}

\caption{Expected Hilbert functions for select random Artinian  algebras $\rmi+(x_1^D,\dots,x_n^D)$, where $\rmi \sim \rmidist(n,D,p)$.  The figure on the left represents algebras generated using the value of $p$ approximately $1/D$, while on the right $1/D^2$ and $1/D^3$. The samples  are selected  from the same samples of monomial algebras from Table~\ref{table:RandomArtinUnimodalStats}. 
}
\label{fig:expectedH random artin algebras - 3}
\end{figure}

\begin{figure}[!h]
Random Artinian  algebras of the form $\rmi+(x_1,\dots,x_n)^D$ for  $n=4, D=50$: 

	\includegraphics[scale=.4]{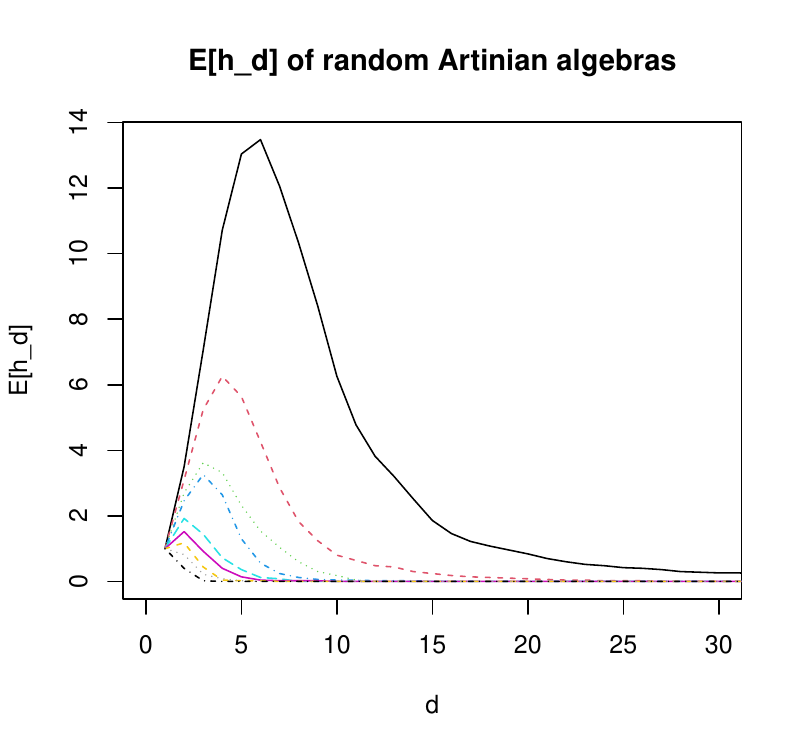}
	\includegraphics[scale=.4]{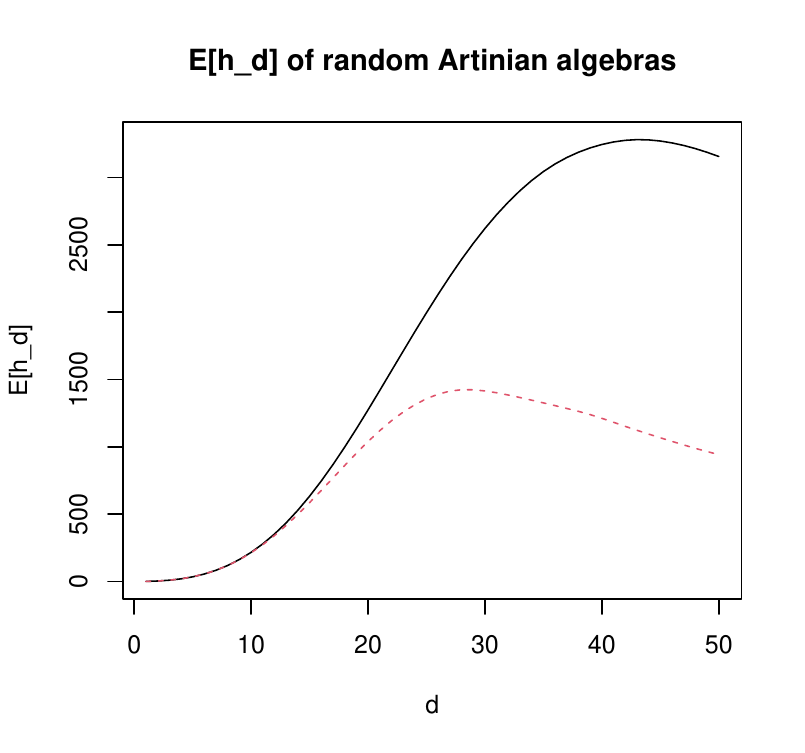}

\caption{Expected Hilbert functions random Artinian algebras $\rmi+(x_1,\dots,x_n)^D$, where $\rmi \sim \rmidist(n,D,p)$.  The figure on the right represents algebras generated using the value of $p$ approximately $1/D$, while on the left the values $p>0.1$. The samples  are selected from the same samples of monomial algebras from Table~\ref{table:RandomArtinPlusUnimodalStats}. 
}
\label{fig:expectedH random artin plus algebras - 1}
\end{figure}

%%%%%%%%%%%

\subsection{WLP simulations} 

%Table~\ref{table:RandomERwlpStats} shows  statistics on WLP for monomial ideals drawn from $\rmidist(n,D,p)$. 
Using the Krull dimension threshold table in  \cite[Figure 1]{dLPSSW}, the parameter $p$ should be set to $p>1/D$ to obtain zero-dimensional ideals with high probability. 
For $n=3,4,5$ variables, when $p=1/D$, we found that %{\color{magenta}what are these 30-60 numbers? check.} 
approximately 30-60\% of sampled zero-dimensional ideals had the WLP, and the rest did not (the failure was more common for smaller values of $p$, such as $p=0.01$ in the case $D=100$). 
This was true for various values of $n$ and $D$ in our simulations. 
As $p=2/D,3/D,...$  moved closer to $1$, the failure of the WLP became a very rare event, as expected. 

Next, we repeated the simulations for Artinian algebras similar to unimodality simulations above. 
Sampled monomial algebras reported in Figures \ref{table:RandomArtinWLPStatsFailure} and \ref{table:RandomArtinPlusWLPStatsFailure} are exactly the same samples of ideals used to compute statistics on unimodality of Hilbert functions  reported in Tables \ref{table:RandomArtinUnimodalStats}  and \ref{table:RandomArtinPlusUnimodalStats}, respectively. 
%Results are summarized in Table \ref{table:RandomArtinWLPStats}. 
%%\begin{table}[!h]
%%	\begin{tabular}{  l l}
%%	$n=3$, $D=50$ &  \\
%%	parameter setting & $\hat{w}$ \\% proportion with WLP \\ 
%%	\hline
%%	\end{tabular}
%%	\caption{Estimated proportion $\hat{w}$ of random  artin algebras $A=S/(I+(x_1^D,\dots,x_n^D))$ for which WLP holds, where $\rmi \sim \rmidist(n,D,p)$.
%%	% .... Sample size for each table entry is $N=100$. This is a typical output from the many simulations we did for $n=3,4,5$ and $D=50,100$. 
%%	Similar rates of (non)failure of WLP hold for the model $I+(x_1,\dots,x_n)^D$ .} 
%% \label{table:RandomArtinWLPStats} 
%%\end{table}
%%The estimated quantity 

We observe an interesting phenomenon: there appears to be  a certain value, namely $p=1/D^2$, for which failure of the WLP is very common! 
%Tables \ref{table:RandomArtinWLPStatsFailure} and \ref{table:RandomArtinPlusWLPStatsFailure} show the results. 

\begin{table}[h]
In the tables below, $\hat\omega$ represents the estimated probability $\omega=\prob{A\mbox{ has the WLP}}$. 

\medskip
 
	\begin{tabular}{  l l}
	$n=3$, $D=50$ &  \\
	parameter setting &  $\hat\omega$ \\% proportion with WLP \\ 
	\hline
	 p=0.000008&  1.00\\ 
	 $p=1/D^2=0.0004$ & 0.13  \\
	 $0.04<p\leq0.1$ & 0.7 \\ 
	 $0.1<p<0.2$ & 0.85\\ 
	$0.2\leq p<0.6$ &0.99\\ 
	$p\geq$ 0.6 & 1 \\
	\hline\hline
	$n=3$, $D=100$&  \\
	\hline
	 p=0.000001&  0.98 \\ 
	 $p=1/D^2=0.0001$ & 0.03  \\
	 p=0.01 & 0.33 \\  
	 $p\geq0.1$ & $>0.9$\\
%	 $0.2\leq p\leq0.4$ & 0.97\\
%	 $p\geq 0.5$ & 1\\
	\end{tabular}
%	
%	\vspace{5mm}
%
\qquad
%	\begin{tabular}{  l l}
%	$n=3$, $D=100$&  \\
%	parameter setting &  $\hat\omega$ \\% proportion with WLP \\ 
%	\hline
%	 p=0.01 & 0.33 \\  
%	 p=0.0001 & 0.03  \\
%	 p=0.000001&  0.98 \\ 
%	\end{tabular}
%\qquad
	\begin{tabular}{  l l}
	$n=4$, $D=50$ &  \\
	parameter setting &  $\hat\omega$ \\% proportion with WLP \\ 
	\hline
	 $p=1/D^4$ &  0.95\\ 
	 $p=1/D^3=0.000008$ & 0.16\\
	 $p=1/D^2=0.0004$ & 0.05 \\
	 $p=0.01$ & 0.05 \\ 
	$p=0.2$ & 0.1\\
	$p=0.3$ & 0.88 \\
	$p=0.4$ & 0.96 \\
	$p=0.5$ & 0.98 \\
	$p\geq 0.6$ & 1\\
	\end{tabular}	
\qquad
	\begin{tabular}{  l l}
	$n=5$, $D=50$ &  \\
	parameter setting &  $\hat\omega$ \\% proportion with WLP \\ 
	\hline
	 $p=1/D^4$ &  0.95\\ 
	 $p=1/D^3=0.000008$ & 0.16\\
	 $p=1/D^2=0.0004$ & 0.01 \\
	 $p=0.02$ & 0.05 \\ 
	$p=0.1$ & 0.22\\
	$p=0.2$ & 0.50 \\
	$p=0.3$ & 0.86 \\
	$0.4\leq p\leq 0.6$ & 0.95 \\	
	$p\geq 0.7$ & 1 \\	
	\end{tabular}	
	\vspace{3mm}

	\caption{Estimated proportion $\hat\omega$ of random  Artinian algebras $A=S/(I+(x_1^D,\dots,x_n^D))$ for which WLP holds, where $\rmi \sim \rmidist(n,D,p)$.
%	A set of samples of random  artin algebras with $I+(x_1^D,\dots,x_n^D)$ for $\rmi \sim \rmidist(n,D,p)$. .... Sample size for each table entry is $N=100$.
%Left: $n=3$. Right: $n=5$. {\bf TRY n=4}
}
 \label{table:RandomArtinWLPStatsFailure} 
\end{table}

\begin{conj}  
\label{conj:failure WLP} 
Consider a random Artinian algebra of the form $\rmi+(x_1^D,\dots,x_n^D)$, where $\rmi \sim \rmidist(n,D,p)$ and $n \ge 3$. If $p=1/D^2$, the probability of $A=S/I$ having the weak Leftschetz property tends to $0$ as $D$ grows: 
\[
	\lim_{D\to\infty} \prob{\mbox{A has the WLP}} \to 0. 
\]
\end{conj}

The  simulation data in Tables~\ref{table:evidence of WLP failure} and \ref{table:RandomArtinPlusWLPStatsFailure}  support the conjecture. 
%--- RANDOM ARTIN ALGEBRAS I_ER + (x_1^D, ..., x_n^D) : 
%-- found a pattern? when D= 100 (not when D=50), around p=1/D^2, it seems that LOW PROB of wlp !! 
%-- let's test it: 
%-- n=3
\begin{table}
	\begin{tabular}{ lll}
	parameter & non-unimodal $h$-vectors &  proportion with WLP\\ 
	$n=3$, $p=1/D^2$&  & \\ 
	\hline
	$D=50$ & 11 &  0.16\\ 
	$D=100$ & 31 &  0.02\\ 
	$D=150$ & 46 &  0\\ 
	$D=200$ & 40 &  0.01\\ 
	\hline\hline
	\end{tabular}	
\qquad
	\begin{tabular}{ lll}
	parameter & non-unimodal $h$-vectors &  proportion with WLP\\ 
	$n=4$, $p=1/D^2$  && \\ 
	\hline
	$D=50$ & 5 &  0.1\\ 
	$D=100$ & 20&  0.03\\ 
	$D=150$ & 39 &  0.01\\ 
	\end{tabular}	
\caption{Evidence of failure of WLP for $n=3$ and increasing $D$, with $p=1/D^2$. Sample size $N=100$ for each parameter setting.} 
\label{table:evidence of WLP failure}
\end{table}

\begin{table}[h]
	\begin{tabular}{  l l}
	$n=3$, $D=50$ &  \\
	parameter setting & $\hat{w}$ \\% proportion with WLP \\ 
	\hline
	$p=1/D^3=0.000008$ & 1\\
	$p=1/D^2=0.0004$ & 0.3\\
	$p=1/D=0.02$ & 0.51 \\
	\\
	$n=3$, $D=100$ &  \\
	\hline
	$p=1/D^3=0.000001$ & 0.28\\
	$p=1/D^2=0.0001$ & 0.02\\
	$p=1/D=0.01$ & 1 \\
	\end{tabular}
\qquad
	\begin{tabular}{  l l}
	$n=4$, $D=20*$ &  \\
	parameter setting &  $\hat{w}$ \\% proportion with WLP \\ 
	\hline
	$p=0.0004$ & 1\\
	$p=1/D^2=0.0025$ & 0.1\\
	$p=0.001$ & 0.9\\
	$p=0.01$ & 0  - 0.1\\
	$p=0.02$ & 0.06\\
	\end{tabular}
\qquad 
	\begin{tabular}{  l l}
	$n=4$, $D=50$ &  \\
	parameter setting & $\hat{w}$ \\% proportion with WLP \\ 
	\hline
	$p=0.1$ & 0.72\\
	$p=0.2$ & 0.76 \\
	$p=0.3$ & 0.88 \\
	$p=0.4$ & 0.96 \\
	$p=0.5$ & 0.98 \\
	$p\geq 0.6$ & 1 \\	
	\end{tabular}
	
	\vspace{3mm}
	\caption{Estimated proportion $\hat{w}$ of random  artin algebras $A=S/(I+(x_1,\dots,x_n)^D)$ for which WLP holds, where $\rmi \sim \rmidist(n,D,p)$. This data was nontrivial to compute, so we are only showing the results for $n=3$ and $n=4$ variables. For very small values of $p<0.02$, the sample size was reduced to $10$ or $20$ in repeated simulations, to allow for reasonable computation time. 
%	Sample size for each table entry is $N=100$.
%Left: $n=3$. Right: $n=5$.
}
 \label{table:RandomArtinPlusWLPStatsFailure} 
\end{table}

%%%%%%%%%%%%%%%%%%%%%%%%%%%%%%%%%%%%%%

\section{Pure  $O$-sequences} 
\label{sec-Pure O} 

There is another model for generating  Artinian monomial ideals. One can pick monomials that generated the inverse system of an algebra. If one chooses all monomial of the same degree, say $D$, then  the corresponding quotient is called a  \emph{monomial level algebra} and its  Hilbert function is called a \emph{pure $O$-sequence}.  Studying pure $O$-sequences is interesting in its own right. We refer to \cite{BMMNZ} for background and further information.

 The main result of this section, \Cref{thm:unimodal level}, shows that the expected Hilbert function of a monomial level algebra is not just unimodal, but even log-concave if one fixes $D$ and a probability  $p$ for choosing monomials of degree $D$ in $S$ in the socle. 
 
% To be a bit more precise, fix $D$ and a probability  $p$ for choosing monomials of degree $D$ in $S$. This will produce a socle of an artinian ideal $I$ of $S$ with an expected number of elements, the \emph{type} of $A = S/I$. Now we consider analogous questions: 
 
 Throughout this section,  $A=S/\rmi$ denotes an Artinian level algebra whose socle is generated by a set $B$ of  monomials of degree $D$.  It follows that $\rmi$ is a monomial ideal. It is  is determined by the socle of $A$ because $I = \Ann (B)$ is the annihilator of $B$ in the sense of Macaulay-Matlis duality. Recall that this annihilator can be computed as follows. For a monomial $x^b = x_1^{b_1} \cdots x_n^{b_n}$, it is well-known that 
 \[
 \Ann (x^b) = \langle x_1^{b_1+1},\ldots,x_n^{b_n +1} \rangle. 
 \]
 Thus, $\Ann (B)$ can be explicitly determined using 
\[
 \Ann (B) = \bigcap_{x^b \in B} \Ann (x^b).  
\]

 \begin{remark}
 (i) The Hilbert function of $A$ is always increasing in degrees $\le \frac{D}{2}$ (see, e.g., \cite{BMMNZ}). 
 
 (ii) For 3 variables, the Hilbert function is unimodal if the type is at most 2 (see \cite{BMMNZ}). This was extended to algebras with type 3 in \cite{B}. 
 
 (iii) For 4 variables, the Hilbert function is unimodal if the type is at most 2 (see \cite{B-15}). 
 \end{remark} 
 
 The existence of non-unimodal pure $O$-sequences is interesting in its own right. It is known that one can have as many peaks as desired if one increases the number of variables (see \cite[Theorem 3.9]{BMMNZ}). However, if one fixes the number of variables, there are again open questions. 
 
 \begin{question} Fix the number of variables to be, say, 3. 
 
 (i) What is the least $D$ such that there is a non-unimodal pure $O$-sequence? 
 
 (ii) What is the least type (not fixing $D$) such that there is a non-unimodal pure $O$-sequence? 
 
 (iii) Combinations of (i) and (ii). 
 \end{question}

We now investigate unimodality  for a randomly generated level algebra.  We use the following model, where 
 we denote by $\mon(n,D)$ the set of all monomials in $n$ variables of degree exactly $D$.

\begin{defn} %[Model 1] 
   \label{def:model 1} 
Let $p\in(0,1)$.  Let $\mathcal{B}$ be a random subset of $\mon(n,D)$ formed by including each element of $\mon(n,D)$ independently with probability $p$.  Set $\mathcal{A} = S/ \Ann \mathcal{B}$ so that  $\mathcal{B}   =\soc{\mathcal{A}}$.
\end{defn}

 One easily sees that the expected type of $A$ is 
\[
p\cdot\binom{D+n-1}{n-1}\approx pD^{n-1}.
\]  

As mentioned above, there are level algebras whose Hilbert functions are not unimodal. However, the following result suggests that there are not too many such algebras.  

\begin{thm} 
   \label{thm:unimodal level} 
Let $A$ be a random level algebra generated as described in \Cref{def:model 1}. 
 Then the expected Hilbert function in degree $j$ with $0 \le j < D$ is 
\[
\expect{h_j}=\binom{j+n-1}{n-1}\left(1-q^{\binom{D-j+n-1}{n-1}}\right), 
\]
and the 
expected $h$-vector $\expect{\underline{h}_A}=(1,\expect{h_1},\expect{h_2},\ldots,)$ of $A$ is log-concave, and so in particular unimodal. 
\end{thm}

To understand the notation, the reader should recall the definition of the expected Hilbert function from page~\pageref{expected Hilbert}: 
Since $h_A(d)$ is a function of a random variable $A=S/I$, for every degree $d$ the quantity $h_A(d)$ is random. For each degree $d$, we can compute the expectation of $h_A(d)$, which we denote by  $\expect{h_d}$. We can also compute the expectation of the $h$-vector for the set of all level algebras generated under the model in \Cref{def:model 1}.  %Model 1. 
This we denote by $\expect{\underline{h}_A}$, where the expectation of the random vector is taken entry-wise.

In order to establish \Cref{thm:unimodal level} we need a preparatory result. The argument is elementary though it requires considerable work. 

\begin{proposition}
   \label{prop:non-negativity} 
Fix integers $k, m \ge 0$.  
Define a real function $f \colon \R \to \R$ by 
\[
f(x) = \big [1 - x^{\binom{m+k+1}{m}}  \big ]^2 -  \big  [1 - x^{\binom{m+k+2}{m}}   \big ]  \cdot  \big [1 - x^{\binom{m+k}{m}}  \big ]. 
\] 
If $0 \le x \le 1$, then $f(x) \ge 0$. 
\end{proposition}

\begin{proof}
A computation shows that 
\begin{align*}
f(x) & = -2x^{\binom{m+k+1}{m}}   + x^{2 \binom{m+k+1}{m}} + x^{\binom{m+k}{m}} 
  + x^{\binom{m+k+2}{m}} - x^{\binom{m+k+2}{m} + \binom{m+k}{m}} \\
  & = x^{\binom{m+k}{m}} g(x), 
\end{align*}
where $g \colon \R \to \R$ is the function defined by 
\[
g(x) = 1 + x^{\binom{m+k+1}{m-1} + \binom{m+k}{m-1}} + x^{\binom{m+k+1}{m} + \binom{m+k}{m-1}} -
2 x^{\binom{m+k}{m-1}} - x^{\binom{m+k+2}{m}}. 
\]
Since $f(1) = 0$, we observe that 
\begin{equation}
\label{eq:g at 1} 
g(1) = 0. 
\end{equation}

Differentiating $g$, we obtain 
\[ 
g'(x)  = x^{\binom{m+k}{m-1}} h(x), 
\]
where $h \colon \R \to \R$ is given by 
\begin{align*}
h (x)  = & \Big [ \binom{m+k+1}{m-1} + \binom{m+k}{m-1} \Big ] x^{\binom{m+k+1}{m-1} }  \\
& + \Big [ \binom{m+k+1}{m} 
+ \binom{m+k}{m-1} \Big ]  x^{\binom{m+k+1}{m} }  \\
& - 2 \binom{m+k}{m-1} - \binom{m+k+2}{m} x^{\binom{m+k}{m} + \binom{m+k+1}{m-1}}. 
\end{align*}
Note that 
\begin{equation}
   \label{eq:h at 1} 
   h(1) =  \binom{m+k+1}{m-1}  +  \binom{m+k+1}{m}  -  \binom{m+k+2}{m} = 0. 
\end{equation}

For the derivative of $h$, we obtain
\begin{align*}
h' (x) = x^{ \binom{m+k+1}{m-1} - 1}  p(x), 
\end{align*}
where $p \colon \R \to \R$ is defined by 
\begin{align*}
p (x) = &  \binom{m+k+1}{m-1} \Big  [ \binom{m+k+1}{m-1} + \binom{m+k}{m-1} \Big ] \\
& + \binom{m+k+1}{m} \Big [ \binom{m+k+1}{m}  + \binom{m+k}{m-1} \Big ] x^{\binom{m+k+1}{m} - \binom{m+k+1}{m-1} } \\
& - \binom{m+k+2}{m} \Big [ \binom{m+k}{m} + \binom{m+k+1}{m-1} \Big ] x^{\binom{m+k}{m}}. 
\end{align*}
Obviously, we have 
\begin{equation}
    \label{eq:p at zero} 
    p(0) > 0. 
\end{equation}
% In order to estimate $p(1)$, we write 
A computation shows that 
\begin{align}
  \label{eq:p(1) as sum} 
 p(1) & = \binom{m+k+1}{m-1}^2 + \binom{m+k+1}{m}^2  + \binom{m+k+2}{m} \binom{m+k}{m-1}  \nonumber  \\
 & \hspace*{.6cm}
     - \binom{m+k+2}{m} \Big [ \binom{m+k}{m} + \binom{m+k+1}{m-1} \Big ] \nonumber  \\
  & = C_1 + C_2
\end{align}
with 
\[
C_1 = \binom{m+k+1}{m}^2 - \binom{m+k+2}{m}  \binom{m+k}{m}
\]
and 
\[
C_2 = \binom{m+k+1}{m-1}^2 - \binom{m+k+2}{m}  \binom{m+k+1}{m-1}. 
\]
Notice that one can rewrite $C_1$ as 
\begin{align*}
C_1 = &  \left [ \frac{(m+k) \cdots (k+3)}{m !} \right ]^2  (m+k+1) (k+2) \big [ (m+k+1) (k+2) - (m+k+2) (k+1)  \big ]. 
\end{align*}
It is straightforward to check that the factors of $C_1$ are all positive, and so we see that 
\[
C_1 > 0. 
\]

Similarly, we get 
\[
C_2 =  \left [ \frac{(m+k) \cdots (k+3)}{(m-2) !} \right ]^2  \cdot \frac{m+k+1}{m-1} \cdot 
\Big [ \frac{m+k+1}{m-1} - \frac{m+k+2}{m}  \Big ]. 
\]
Again, one checks that the factors of $C_2$ are all positive, and thus
\[
C_2 > 0. 
\]
Combined with Equation \eqref{eq:p(1) as sum}, the last two estimates show that 
\begin{equation}
   \label{eq: p at one}
p(1) > 0. 
\end{equation}

Yet another differentiation gives
\[
p' (x) = x^{\binom{m+k+1}{m} - \binom{m+k+1}{m-1} -1 } q(x), 
\]
where $q \colon \R \to \R$ is defined by 
\begin{align*}
q(x) = &  \binom{m+k+1}{m} \Big [ \binom{m+k+1}{m}  + \binom{m+k}{m-1} \Big ] \Big [  \binom{m+k+1}{m} - \binom{m+k+1}{m-1} \Big ] \\
& - \binom{m+k+2}{m} \Big [ \binom{m+k}{m} + \binom{m+k+1}{m-1} \Big ] \binom{m+k}{m} x^{\binom{m+k+1}{m-1} - \binom{m+k}{m-1}} . 
\end{align*}
Clearly, we have 
\[
q(0) > 0. 
\]
Moreover, comparing factors in the same position 
%and using 
%\[
% \binom{m+k+1}{m}  + \binom{m+k}{m-1}  - \binom{m+k}{m} -  \binom{m+k+1}{m-1} = 
%\]
 we conclude that 
\begin{align*}
q(1) & =  \binom{m+k+1}{m} \Big [ \binom{m+k+1}{m}  + \binom{m+k}{m-1} \Big ] \Big [  \binom{m+k+1}{m} - \binom{m+k+1}{m-1}  \Big ] \\
& \hspace*{.6cm}  - \Big [ \binom{m+k}{m} + \binom{m+k+1}{m-1} \Big ] \binom{m+k+2}{m}  \binom{m+k}{m} \\
& < 0. 
\end{align*}
Since $q$ is a decreasing function, it follows that it has exactly one zero, $x_0$, in the closed interval $[0,1]$, and so we get that 
\begin{align*} 
q (x) & \ge 0 \; \text{ if } x \in [0, x_0], \; \text{ and } \\
q (x) & \le 0 \; \text{ if } x \in [x_0, 1]. 
\end{align*}
Recall that $p' x)$ is a product of $q(x)$ and a power of $x$. Thus, we see that $q$ is increasing on the interval  $[0, x_0]$ and decreasing on $[x_0, 1]$. Using that $p(0) > 0$ and $p (1) > 0$ (see Inequalities \eqref{eq:p at zero}  and \eqref{eq: p at one}), it follows that $p$ is positive on the interval $[0,1]$.  
As $h' (x)$ is a product of $p(x)$ and a power of $x$, we conclude that $p$ is increasing on $[0,1]$. Hence $h(1) = 0$ (see Equation \eqref{eq:h at 1}) implies $h(x) \le 0$ if $x \in [0, 1]$. 
Now, $g' (x)$ is a product of $p(x)$ and a power of $x$. So,  we conclude that $g$ is decreasing on the interval $[0,1]$. By Equation \eqref{eq:g at 1}, we know that $g (1) = 0$. It follows that $g$ is nonnegative on $[0,1]$, and so the same is true for $f$ as well, which completes the argument. 
\end{proof}

The main result follows now easily: 

\begin{proof}[Proof of \Cref{thm:unimodal level}]  
Fix an integer $d$ with $1\le d\le D$.  For each monomial $x^{\alpha}\in S_d$,  let $1_{\alpha}$ be the indicator random variable for the event $x^{\alpha}\not\in\rmi$ where $A=S/\rmi$.  
Then, $1_{\alpha} \neq 1$ if and only if every monomial $x^{\alpha} x^{\beta}$ with $x^{\beta} \in [S]_{D_d}$ is not in $Soc(A)$. There are $\binom{D-d+n-1}{n-1}$ monomials of degree $D-d$ in $S$. Hence none of the monomials $x^{\alpha} x^{\beta}$ is in $Soc (A)$ with probability $q^{\binom{D-d+n-1}{n-1}}$. It follows that 
\[
\expect{ 1_{\alpha}} = \prob{1_{\alpha}=1}=1-q^{\binom{D-d+n-1}{n-1}}.
\]  
Since $h_d=\sum_{x^{\alpha}\in S_d} 1_{\alpha}$, by linearity of expectation we find that
\[
\expect{h_d}=\binom{d+n-1}{n-1}\left(1-q^{\binom{D-d+n-1}{n-1}} \right), 
\]  
as claimed. 

Thus,  log-concavity of the expected $h$-vector follows from \Cref{prop:non-negativity} with $m = n-1$ and $k = D-d$. 
\end{proof}

%%%%%%%%%%%%%%%%%%%%%%%%%%%%%%%%

\section{WLP for Random level algebras}
  \label{sec:WLP level} 

In this section, we investigate when the WLP holds for a randomly generated level algebra.  We consider two random models.  In each case, we produce a random level algebra $\mathcal{A}$ by choosing a random collection of monomials to be $\soc{A}$. As above, we denote by $\mon(n,D)$ the set of all monomials in $n$ variables of degree exactly $D$.

\begin{defn}[Model 1] 
    \label{def:level model 1}
Let $p\in(0,1)$.  Let $\mathcal{B}$ be a non-empty random subset of $\mon(n,D)$ formed by including each element of $\mon(n,D)$ independently with probability $p$.  Set $\mathcal{B}=\soc{\mathcal{A}}$.
\end{defn}

As pointed out in the previous section,  the random set $\mathcal{B}$ uniquely determines the monomial level algebra $\mathcal{A}$.  Under this model, for fixed $B\subset \mon(n,D)$, one has
\[
\prob{\soc{\mathcal{A}}=B}=p^{|B|}q^{\binom{n-1+D}{n-1}-|B|}.
\]

One can  think of another natural model for randomly generating monomial level algebras, namely the following. 

\begin{defn}[Model 2]  Fix positive integers $D$ and $t$.  Choose $\mathcal{B}\subset\mon(n,D)$ with $|\mathcal{B}|$ uniformly among all $t$-subsets of $\mon(n,D)$.  Set $\mathcal{B}=\soc{\mathcal{A}}$.
\end{defn}
%{\bf Added 13 sep 2023:} in random graphs, the two models are shown to be equivalent because expected number of monomials under model 1 is exactly the number in model 2; but that's graphs not monomial algebras. I have no code that runs Model 2 yet. I have only run simulations with model 1. Is that OK?

Under Model 2, for fixed $B\subset \mon(n,D)$ with $|B|=t$, 
\[
\prob{\soc{\mathcal{\mathcal{A}}}=B}=\binom{\binom{n-1+D}{n-1}}{t}^{-1}.
\]  Thus, in this model, $\mathcal{A}$ is chosen uniformly among all codimension $n$ Artinian level algebras of socle degree $D$ and type $t$. 

In this paper, we do not simulate from Model 2, but leave it for further exploration due to the following folklore fact in random graph theory: \er\ \cite{ER} defined the random graph model on $n$ vertices using a parameter $t$ for the fixed number of edges in the graph. Gilbert \cite{Gilbert} defined the random graph model on $n$ vertices using a probability parameter $p$ for each edge being present or absent, independently of other edges. There is a heuristic equivalence between these two models, based on the law of large numbers; namely that if $pn^2\to\infty$ then the random graphs from the \er\ model should behave similarly to the random graphs from  Gilbert's model. The graph theory literatures shows this to be true for many graph properties. It would be interesting to explore the implications of this to algebraic properties of random monomial ideals generated by the two models.

\medskip 

Compared to unimodality of Hilbert functions, less is known when a monomial level algebra has the WLP.  Note though if $A$ has socle type one then it is a complete intersection and has the WLP thanks to a result of \cite{stanley}, \cite{watanabe} and
\cite{RRR}.
However, there are algebras with socle type two that do not have the WLP. 

\begin{example}[{\cite[Proposition 7.13]{BMMNZ}}]
Assume $n$ is even and $N \ge 5$. Consider 
\[
B = \{ x_1^{N-3} x_2^{N-1} x_3^{N-2} \cdots x_n^{N-2}, 
x_1^{N-1} x_2^{N-3} x_3^{N-2} \cdots x_n^{N-2}\}. 
\]
Then $A = S/ \Ann (B)$ does not have the WLP.  
\end{example}

Again, one wonders how common level algebras are that fail the WLP. In order to quantify such algebras, a first step could be to investigate if threshold results are true.

\begin{question} 
   \label{quest:threshold}
Fix $n \geq 3$ and consider probability $p = p(D)$ as $D \to \infty$: 

(i) Is it true that $\lim_{D \to \infty} \prob {A \text{ has WLP}} = 1$? 

(ii) Are there thresholds, i.e., functions $f(D), g(D)$ such that 
\[
\lim_{D \to \infty}  \prob {A \text{ has WLP}} = 1 
\]
if $p(D) \ll f(D)$, i.e., $ \lim_{D \to \infty} \frac{p(D)}{f(D)} = 0$, or $p(D) \gg g(D)$, i.e., $ \lim_{D \to \infty} \frac{g(D)}{p(D)} = 0$? 
\end{question}

We need some preparation in order to give some answer to Question (ii). 

\begin{lm}
\label{lem:almost full socle} 
Let $A$ be a random level algebra generated as described in \Cref{def:level model 1} such that $\soc  A$ is generated by at least $\binom{n-1+d}{n-1} - 1$ elements. Then $A$ has the weak Lefschetz property. 
\end{lm}

\begin{proof}
Denote the maximal homogenous ideal of $S$ by $\fm$. There are two cases for the size of $\soc  A$. 
If $\dim_{\mathbbm{k}} \soc  A = \binom{n-1+d}{n-1}$ then $A = S/\fm^{D+1}$, and  it follows that $A$ has the weak Lefschetz property. 
Assume that $\dim_{\mathbbm{k}} \soc A = \binom{n-1+d}{n-1} -1$, that is, all but one monomial of $\mon(n,D)$ are in the socle of $A$. Denote the missing monomial by $x^b$ and set $I = \Ann (\soc A)$. Then one checks that $\langle \fm^{D+1}, x^b \rangle \subseteq I$. Moreover, the inverse system generated by $\soc A$ contains $\mon (n, D-1)$. It follows that $A = S/I$ and $S/\langle \fm^{D+1}, x^b \rangle$ have the same Hilbert function, which implies 
\[
I = \langle \fm^{D+1}, x^b \rangle. 
\]

For any $j \ge 0$, consider the map $\varphi_j \colon [A]_{j-1} \to [A]_j$, defined by multiplication by $\ell = x_1 + \cdots + x_n$ on $A$. Since $[A]_j = [S]_j$ if $j < D$, the map $\varphi_j$ is injective in these cases. Clearly, $\varphi_{D+1}$ is surjective. We claim that $\varphi_{D}$ is injective. Indeed, consider any polynomial $g$ in $\ker \varphi_{D}$. Then one has $g \cdot \ell \in I$, and so $g \cdot \ell \in \langle x^b \rangle$. This gives 
\[
g \in \langle x^b \rangle : \ell = \langle x^b \rangle. 
\]
Since $\deg g = D-1 < D = \deg x^b$, we conclude $g = 0$, as claimed. Thus, we have shown that every map $\varphi_j$ has maximal rank, as desired. 
\end{proof}

This has the following consequence with respect to \Cref{quest:threshold}. 

\begin{cor}  
\begin{itemize}

\item[(i)] If $p (D) \ge  1 - (n-1)! D^{-(n-1)}$, 
then 
$\lim_{D \to \infty} \prob {A \text{ has WLP}} = 1$. 

\item[(ii)] If $p(D) \le (n-1)! D^{-(n-1)}$, then $\lim_{D \to \infty}  \prob {A \text{ has WLP}} = 1$.

\end{itemize}
\end{cor}

\begin{proof}
The expected number of monomials of $\mon(n, D)$ in $\soc A$  is $\binom{n-1+D}{n-1} \cdot p(D)$. Hence the assumption in (i) implies that, as $D$ approaches infinity, the socle of $A$ is generated by at least $\binom{n-1+D}{n-1} - 1$ monomials, and so \Cref{lem:almost full socle}  shows that $A$ has the weak Lefschetz property. 

For (ii), we argue similarly. This time, the assumption implies that, as $D$ approaches infinity, the socle of $A$ is generated by at most one element. Since we are ignoring the case where no monomial is chosen, it follows that $A$ is a monomial complete intersection. It is a classical result that $A$ has the weak Lefschetz property in this case, see  \cite{stanley, watanabe, RRR} for different arguments. 
\end{proof}

We do not expect the above threshold functions to be optional and leave improvements for further investigations.

%%%

\subsection{Simulation results}
Finally, we simulate random level algebras as described in \Cref{def:level model 1}. In the  Table~\ref{table:RandomLevelWLPStatsFailure}, $\hat\omega$ represents the estimated probability  $\omega=\prob{A\mbox{ has the WLP}}$ for a random level algebra $A$. 

%-- RANDOM LEVEL ALGEBRAS via socle : 
%rate of sample with wlp = .52 -- for D50 n3 p.02
%rate of sample with wlp = 1 -- for D50 n3 p.0004
%rate of sample with wlp = 1 -- for D50 n3 p.000008
%rate of sample with wlp = .38 -- for D100 n3 p.01
%rate of sample with wlp = 1 -- for D100 n3 p.0001
%rate of sample with wlp = 1 -- for D100 n3 p.000001

\begin{table}[h]
	\begin{tabular}{  l | l}
	$n=3$, $D=50$ &  \\
	\hline
	parameter & $\hat{w}$ \\% proportion with WLP \\ 
	\hline
	 $p=0.000008$ &  1.00\\ 
	 $p=0.0004$ & 1.00  \\
	 $p=0.02$ & 0.52 \\  
	 $p=0.1$ &0.65\\
	 $p=0.2$ &0.87\\
	 $p=0.3$ &0.93\\
	 $p=0.4$ &0.66\\ % !!!!! 
	 $p=0.5$ &0.98\\ 
	 $p=0.6$ &0.79\\ 
	 $p=0.7$ &0.98\\ 
	 $p=0.8$ &1\\ 
	 $p=0.9$ &1\\ 
	\end{tabular}
\qquad 
	\begin{tabular}{  l | l}
	$n=3$, $D=100$&  \\
	\hline
	parameter  & $\hat{w}$ \\% proportion with WLP \\ 
	\hline
	 $p=0.000001$ &  1.0 \\ 
	 $p=0.0001$ & 1.0  \\
	 $p=0.01$ & 0.38 \\  
	 $p=0.1$ &0.64\\
	 $p=0.2$ &0.58\\
	 $p=0.3$ &0.48\\
	 $p=0.4$ &0.84\\
	 $p=0.5$ &0.96\\ 
	 $p=0.6$ &1\\ 
	 $p=0.7$ &0.7\\ 
	 $p=0.8$ &1\\ 
	 $p=0.9$ &1\\ 
	\end{tabular}
	\vspace{3mm}

	\caption{Estimated proportion $\hat{w}$ of  random level algebras, generated using \Cref{def:level model 1}, that have the weak Leftschetz property. Sample size for each table entry is $N=100$ when $D=50$ and $N=50$ when $D=100$.
	}
 \label{table:RandomLevelWLPStatsFailure} 
\end{table}

The reader might note that, for example in the case $n=3$ and $D=50$, we do not see  a clear pattern similar to random Artinian algebras. Namely there is not a clear threshold for emergence of the WLP with high probability, because the proportion of the sample with WLP varies with $p$ in ways we do not yet understand!

 %%%%%%%%%%%%%%%%%%%%%
%BIBLIOGRAPHY--   {\color{magenta}  - why are we not using .bib files? see randomized-ideals.bib} 
%\bibliographystyle{natbib}

\section*{Acknowledgements}
The authors are grateful for Dane Wilburne's initial simulations for this project when we started the investigation of unimodality and the WLP in random algebras. Dane was a PhD student at Illinois Tech at that time. 

UN was partially supported by Simons Foundation grant \#636513. SP was partially supported by the Simons Foundation's Travel Support for Mathematicians Gift 854770 and  DOE/SC award number \#1010629. During the final stages of this project, SP was in residence at the program `Algebraic Statistics and Our Changing World' hosted by the Institute for Mathematical and Statistical Innovation (IMSI), which is supported by the National Science Foundation (Grant No. DMS-1929348).

\end{document}